\documentclass{amsart}

\usepackage{amssymb,amsmath,amsthm,latexsym,booktabs} %,color,relsize}

\theoremstyle{definition}
\newtheorem{definition}{Definition}[section]
\newtheorem{example}[definition]{Example}
\newtheorem{remark}[definition]{Remark}

\theoremstyle{plain}
\newtheorem{lemma}[definition]{Lemma}
\newtheorem{proposition}[definition]{Proposition}
\newtheorem{theorem}[definition]{Theorem}

\newcommand{\field}{\mathbb{K}}

\allowdisplaybreaks

\begin{document}

\title[Lie-Yamaguti algebras]
{Symmetric matrices, orthogonal Lie algebras, and Lie-Yamaguti algebras}

\author[Benito]{Pilar Benito}

\address{Departamento de Matem\'aticas y Computaci\'on, Universidad de La Rioja, Espa\~na}

\email{pilar.benito@unirioja.es}

\author[Bremner]{Murray Bremner}

\address{Department of Mathematics and Statistics, University of Saskatchewan, Canada}

\email{bremner@math.usask.ca}

\author[Madariaga]{Sara Madariaga}

\address{Department of Mathematics and Statistics, University of Saskatchewan, Canada}

\email{madariaga@math.usask.ca}

%\date{???}



\begin{abstract}
On the set $H_n(\field)$ of symmetric $n \times n$ matrices over the field $\field$ we can define
various binary and ternary products which endow it with the structure of a Jordan algebra or a Lie
or Jordan triple system.
All these non-associative structures have the orthogonal Lie algebra $\mathfrak{so}(n,\field)$ as
derivation algebra.
This gives an embedding $\mathfrak{so}(n,\field) \subset \mathfrak{so}(N,\field)$ for $N = \binom{n+1}{2}-1$.
We obtain a sequence of reductive pairs $(\mathfrak{so}(N,\field), \mathfrak{so}(n,\field))$ that provides
a family of irreducible Lie-Yamaguti algebras.
In this paper we explain in detail the construction of these Lie-Yamaguti algebras.
In the cases $n \le 4$, we use computer algebra to determine the polynomial identities of degree $\le 6$;
we also study the identities relating the bilinear Lie-Yamaguti product with the trilinear product obtained from
the Jordan triple product.
\end{abstract}

\maketitle

%%%%%%%%%%%%%%%%%%%%%%%%%%%%%%%%%%%%%%%%%%%%%%%%%%%%%%%%%%%%%%%%%%%%%%%%%%%%%%%%%%%%%%%%%%%%%%%%%%%%%%%%%%%%
%%%%%%%%%%%%%%%%%%%%%%%%%%%%%%%%%%%%%%%%%%%%%%%%%%%%%%%%%%%%%%%%%%%%%%%%%%%%%%%%%%%%%%%%%%%%%%%%%%%%%%%%%%%%
%%%%%%%%%%%%%%%%%%%%%%%%%%%%%%%%%%%%%%%%%%%%%%%%%%%%%%%%%%%%%%%%%%%%%%%%%%%%%%%%%%%%%%%%%%%%%%%%%%%%%%%%%%%%

\section{Introduction}

In this paper we use the representation theory of Lie algebras and computer algebra to study a family of
Lie-Yamaguti algebras whose standard envelopes are orthogonal simple Lie algebras.
Our main goal is to determine the polynomial identities satisfied by these structures.

\begin{definition} \label{df:LY} \cite{K,KW,Ya58}
A \emph{Lie-Yamaguti algebra} (\emph{LY algebra} for short)
%$(\mathfrak{m},x\cdot y,\langle x\,,y\,,z\,\rangle)$
is a vector space $\mathfrak{m}$ with a bilinear product $\mathfrak{m}\times\mathfrak{m}\rightarrow \mathfrak{m}$,
$(x,y) \mapsto x \cdot y$, and a trilinear product
$\mathfrak{m}\times\mathfrak{m}\times\mathfrak{m}\rightarrow \mathfrak{m}$, $(x,y,z) \mapsto \langle x, y, z \rangle$,
satisfying:
\begin{align}
\tag{LY1} & x\cdot x \equiv 0 \\
\tag{LY2} & \langle x,x,y\rangle \equiv 0 \\
\tag{LY3} & \circlearrowleft_{x,y,z} (\langle x,y,z\rangle+(x\cdot y)\cdot z ) \equiv 0 \\
\tag{LY4} & \circlearrowleft_{x,y,z} \langle x\cdot y,z,t\rangle \equiv 0 \\
\tag{LY5} & \langle x, y,u\cdot v\rangle \equiv \langle x,y,u\rangle\cdot v+u\cdot \langle x,y,v\rangle \\
\tag{LY6} & \langle x,y,\langle u,v,w\rangle\rangle \equiv
\langle\langle x,y,u\rangle,v,w\rangle+\langle u,\langle x,y,v\rangle,w\rangle+\langle u,v,\langle x,y,w\rangle\rangle
\end{align}
Here $\circlearrowleft_{x,y,z}$ means the cyclic sum on $x,y,z$, and $\equiv$ indicates that the equation holds
for all values of the variables.
\end{definition}

LY algebras with trivial ternary product are Lie algebras.
LY algebras with trivial binary product are Lie triple systems (LTS), which appear in the study of the symmetric spaces.
They are the odd parts of $\mathbb{Z}_2$-graded Lie algebras which are simple as graded algebras,
and were classified in \cite{Lister} using involutive automorphisms;
see \cite{MH} for an approach using representation theory.
A simple LTS is either irreducible or 2-irreducible as a module over its derivation Lie algebra;
in the 2-irreducible case, the components are dual modules.
For general simple LY algebras, the structure is more complicated \cite{BDE05}.
One goal of the present paper is to develop interesting examples to clarify the general theory.

Many examples of LY algebras are provided by reductive pairs \cite{Sa68} which are closely related to
reductive homogeneous spaces; in particular, LY algebras related to spheres have been studied in \cite{EM1,EM2}.
The LY algebras whose standard envelope is a Lie algebra of type $G_2$ are studied in \cite{BDE05};
several new examples of binary-ternary products are given in terms of the octonions,
and all nonabelian reductive subalgebras of $G_2$ are classified up to conjugation over a field of characteristic 0.
In that paper, the associated reductive pairs and LY algebras were described in detail,
using the representation theory of the split 3-dimensional simple Lie algebra in the spirit of \cite{Di84}.

The paper is organized as follows.
Section \ref{sectionprelim} gives some basic definitions and constructions.
Section \ref{sectionsymmetric} introduces the reductive pairs related to orthogonal Lie algebras, and the corresponding LY algebras.
Section \ref{sectionnonassoc} provides some combinatorial and algorithmic background on polynomial identities.
Section \ref{sectionresults} presents our computational results for the LY algebra obtained from the embedding
$\mathfrak{so}(4) \subset \mathfrak{so}(9)$.
Section \ref{sectionjordan} studies the binary-ternary structure obtained by replacing the trilinear LY operation
with the Jordan triple product.
Section \ref{sectionconclusion} gives some suggestions for further research.

%%%%%%%%%%%%%%%%%%%%%%%%%%%%%%%%%%%%%%%%%%%%%%%%%%%%%%%%%%%%%%%%%%%%%%%%%%%%%%%%%%%%%%%%%%%%%%%%%%%%%%%%%%%%
%%%%%%%%%%%%%%%%%%%%%%%%%%%%%%%%%%%%%%%%%%%%%%%%%%%%%%%%%%%%%%%%%%%%%%%%%%%%%%%%%%%%%%%%%%%%%%%%%%%%%%%%%%%%
%%%%%%%%%%%%%%%%%%%%%%%%%%%%%%%%%%%%%%%%%%%%%%%%%%%%%%%%%%%%%%%%%%%%%%%%%%%%%%%%%%%%%%%%%%%%%%%%%%%%%%%%%%%%

\section{Preliminaries} \label{sectionprelim}

Following \cite{BEM09,SaWi67}, we note that two elements $x,y$ in an LY algebra $\mathfrak{m}$ define
a linear map $d_{x,y}\colon \mathfrak{m}\rightarrow\mathfrak{m}$, $z \mapsto \langle x,y,z\rangle$,
which by (LY5-6) is a derivation of both the binary and ternary products.

\begin{definition}
The linear maps $d_{x,y}$ are called \emph{inner derivations}.
We write $D(\mathfrak{m}) \subseteq \mathrm{End}(\mathfrak{m})$ for their linear span.
\end{definition}

\begin{remark}
By (LY6), the space $D(\mathfrak{m})$ is closed under the Lie bracket.
\end{remark}

\begin{definition}
The \emph{standard envelope} of $\mathfrak{m}$ is the direct sum
$\mathfrak{g}(\mathfrak{m}) = D(\mathfrak{m}) \oplus \mathfrak{m}$ endowed with the structure of a Lie algebra by
\begin{align*}
&
[D(x,y),D(z,t)]= D([x,y,z],t)+D(z,[x,y,t]),
\\
&
[D(x,y),z]=D(x,y)(z)=[x,y,z],
\qquad
[z,t]=D(z,t)+z\cdot t.
\end{align*}
\end{definition}

\begin{remark}
$D(\mathfrak{m})$ is a subalgebra of $\mathfrak{g}(\mathfrak{m})$, and $\mathfrak{m}$ is an $D(\mathfrak{m})$-module.
\end{remark}

\begin{definition}
Let $\mathfrak{g}$ be a Lie algebra with product $[x,y]$.
A \emph{reductive decomposition} of $\mathfrak{g}$ is a vector space direct sum
$\mathfrak{g} = \mathfrak{h} \oplus \mathfrak{m}$ satisfying
$[\mathfrak{h},\mathfrak{h}]\subseteq \mathfrak{h}$ and $[\mathfrak{h},\mathfrak{m}]\subseteq \mathfrak{m}$.
In this case, we call $(\mathfrak{g},\mathfrak{h})$ a \emph{reductive pair}.
\end{definition}

For a reductive decomposition $\mathfrak{g}=\mathfrak{h}\oplus \mathfrak{m}$,
there exist natural binary and ternary products on $\mathfrak{m}$ defined by
\begin{equation}
\label{standardLY}
x\cdot y = \pi_{\mathfrak{m}}\bigl([x,y]\bigr),
\qquad
[x,y,z]=\bigl[\pi_{\mathfrak{h}}([x,y]),z],
\end{equation}
where $\pi_{\mathfrak{h}}$ and $\pi_{\mathfrak{m}}$ are the projections on $\mathfrak{h}$ and $\mathfrak{m}$.
These products endow $\mathfrak{m}$ with the structure of an LY algebra, and
any LY algebra has this form since $( \, \mathfrak{g}(\mathfrak{m}), D(\mathfrak{m}) \, )$
is a reductive pair.

\begin{definition}
The LY algebra structure \eqref{standardLY} is the \emph{standard LY algebra} given by the reductive pair
$(\mathfrak{g},\mathfrak{h})$.
We call this LY algebra \emph{irreducible} if $\mathfrak{m}$ is an irreducible $\mathfrak{h}$-module,
and $k$-\emph{irreducible} if $\mathfrak{m}$ is the direct sum of $k$
irreducible $\mathfrak{h}$-modules.
\end{definition}

The classification of isotropically irreducible reductive homogeneous spaces by Wolf \cite{Wo68,Wo84}
was reconsidered in \cite{BEM09,BEM11} using an algebraic approach with reductive pairs related to
nonassociative structures, in particular Lie and Jordan algebras and triples.
A summary of the close connections between nonassociative structures and isotropically irreducible
reductive homogeneous spaces is given in \cite[Tables 7--9]{BEM11}.
We emphasize \cite[Ch.~I, \S 11]{Wo68} the two sequences of homogeneous spaces
  \[
  \mathbf{SO}(2k^2{+}3k)/ \, \mathbf{SO}(2k{+}1) \quad (k \ge 2),
  \qquad
  \mathbf{SO}(2k^2{+}k{-}1) / \, \mathbf{SO}(2k) \quad (k \ge 4).
  \]
If we study the corresponding Lie algebras, we can combine these two sequences into a single sequence of
orthogonal reductive pairs,
  \[
  (\mathfrak{so}(N), \mathfrak{so}(n))
  \qquad
  n \ge 3, \; n \ne 4,
  \qquad
  N = \binom{n{+}1}{2}-1.
  \]
We have included two additional cases, one from each sequence of homogeneous spaces.
For $n = 3$ (first sequence, $k = 1$) we obtain $\mathbf{SO}(5)/ \, \mathbf{SO}(3)$,
% which we regard as $\mathbf{Sp}(4)/\mathbf{SO}(3)$.
and for $n = 6$ (second sequence, $k = 3$) we obtain $\mathbf{SO}(20) / \, \mathbf{SO}(6)$.
% which we regard as $\mathbf{SO}(20)/\mathbf{SU}(4)$.
% The Lie algebras $\mathfrak{sp}(4)$ and $\mathfrak{su}(4)$ are simple of
% type $C_2 \approx B_2$ and $A_3 \approx D_3$; over an algebraically closed field, they can be regarded as
% $\mathfrak{so}(5)$ and $\mathfrak{so}(6)$.
The case $n = 2$ (second sequence, $k = 1$) is trivial.
The case $n = 4$ (second sequence, $k = 2$) is a 2-irreducible LY algebra related to the homogeneous space
$\mathbf{SO}(9)/ (\mathbf{SO}(3)\times\mathbf{SO}(3))$.
(See Dickinson and Kerr \cite{Dic08} for general results on compact homogeneous spaces with two isotropy summands.)

All structures in this paper are finite-dimensional over an algebraically closed field $\field$ of characteristic 0.

%%%%%%%%%%%%%%%%%%%%%%%%%%%%%%%%%%%%%%%%%%%%%%%%%%%%%%%%%%%%%%%%%%%%%%%%%%%%%%%%%%%%%%%%%%%%%%%%%%%%%%%%%%%%
%%%%%%%%%%%%%%%%%%%%%%%%%%%%%%%%%%%%%%%%%%%%%%%%%%%%%%%%%%%%%%%%%%%%%%%%%%%%%%%%%%%%%%%%%%%%%%%%%%%%%%%%%%%%
%%%%%%%%%%%%%%%%%%%%%%%%%%%%%%%%%%%%%%%%%%%%%%%%%%%%%%%%%%%%%%%%%%%%%%%%%%%%%%%%%%%%%%%%%%%%%%%%%%%%%%%%%%%%

\section{LY algebras related to symmetric matrices} \label{sectionsymmetric}

The vector space $H_n(\field)$ of symmetric $n\times n$ matrices can be endowed with binary and ternary products
giving a Jordan algebra, a Lie triple system or a Jordan triple system.
Each of these non-associative structures has the orthogonal Lie algebra $\mathfrak{so}_n(\field)$ of type $B$ or $D$
as its Lie algebra of derivations $\mathrm{Der}\,H_n(\field)$; this Lie algebra is simple except for $n=2,4$.
We denote by $H_n(\field)_0$ the subspace of matrices with trace 0; this is the orthogonal complement of
the 1-dimensional subspace spanned by the identity matrix $I_n$ with respect to the trace form $T(a,b) = \mathrm{tr}(ab)$.
We denote the dimension of $H_n(\field)_0$ by $N = \binom{n+1}{2}-1$.
Every derivation $d \in \mathrm{Der}\,H_n(\field)$ annihilates $I_n$ and satisfies $d( H_n(\field)) \subset H_n(\field)_0$.
Since the restriction of $T$ to $H_n(\field)_0$ is nondegenerate, $\mathrm{Der}\, H_n(\field)$ embeds into
the Lie algebra $\mathfrak{so}_N(\field)$ of endomorphisms of $H_n(\field)_0$ which are orthogonal with respect to $T$.
For $n \ge 2$ we obtain a sequence of reductive pairs $(\mathfrak{so}(H_n(\field)_0, T), \mathrm{Der}\, H_n(\field))$
which is the Lie algebraic analogue of the reductive homogeneous spaces mentioned in the Introduction as quotients of
orthogonal Lie groups.
(Since $n = 2$ gives $N = 2$, we ignore this trivial case.)
The components of these reductive pairs are simple Lie algebras for $n \ge 3$, $n \ne 4$;
this implies some preliminary structural results on the associated LY algebras
(see \cite[Th.~2]{SaWi67} and \cite[Pr.~1.3]{BEM09}):

\begin{theorem}
If $\mathfrak{g}$ is a simple Lie algebra and $\mathfrak{h} \subset \mathfrak{g}$ is a semisimple subalgebra,
then $(\mathfrak{g},\mathfrak{h})$ is a reductive pair with $\mathfrak{m} = \mathfrak{h}^\perp$ with respect
to the Killing form, and either $[\mathfrak{m},\mathfrak{m}] = \{0\}$ or $\mathfrak{m}$ is simple.
\end{theorem}

\begin{proposition} \label{pr:envueltasimple}
Let $\mathfrak{g} = \mathfrak{h} \oplus \mathfrak{m}$ be a reductive decomposition of a simple Lie algebra with
$\mathfrak{m} \ne \{0\}$.
Then for the LY algebra structure on $\mathfrak{m}$ defined by \eqref{standardLY}, we have
$\mathfrak{h} \cong D(\mathfrak{m})$ (the inner derivation algebra) and
$\mathfrak{g} \cong D(\mathfrak{m}) \oplus \mathfrak{m}$ (the standard envelope).
Furthermore, if $\mathfrak{h}$ is semisimple and $\mathfrak{m}$ is an irreducible $\mathfrak{h}$-module,
then either $\mathfrak{h} \cong \mathfrak{m}$ as $\mathrm{ad}_\mathfrak{h}$-modules, or
$\mathfrak{m} = \mathfrak{h}^\perp$ with respect to the Killing form of $\mathfrak{g}$.
\end{proposition}

In the rest of this section we explain different constructions of the reductive pairs
$\big(\mathfrak{so}(H_n(\field)_0,T), \mathrm{Der}\,H_n(\field) \big)$ and the associated LY algebras
using representation theory of Lie algebras and structure theory of Jordan algebras.

\subsection{LY algebras from the viewpoint of representation theory} \label{subsection:LY-representation}

Let $\mathfrak{h}$ be a simple Lie algebra and let $\mathfrak{m}$ be an irreducible $\mathfrak{h}$-module.
If the exterior square $\Lambda^2 \mathfrak{m}$ contains both $\mathfrak{m}$ and the adjoint $\mathfrak{h}$-module,
then there are $\mathfrak{h}$-module morphisms
  \[
  \beta\colon \Lambda^2 \mathfrak{m} \longrightarrow \mathfrak{m},
  \qquad
  \tau\colon \Lambda^2 \mathfrak{m} \longrightarrow \mathfrak{h},
  \]
which give bilinear and trilinear products on $\mathfrak{m}$ defined by
  \[
  [x,y] = \beta(x,y),
  \qquad
  \langle x,y,z \rangle = \tau(x,y) \cdot z,
  \]
where $\cdot$ is the action of $\mathfrak{h}$ on $\mathfrak{m}$.
Since $\beta$ and $\tau$ have domain $\Lambda^2 \mathfrak{m}$, these products satisfy
the skew-symmetries (LY1) and (LY2) in the definition of an LY algebra:
  \[
  [x,x] = 0,
  \qquad
  \langle x,x,y \rangle = 0.
  \]
Identities (LY5) and (LY6) hold because $\beta$ and $\tau$ are $\mathfrak{h}$-module morphisms,
so $\mathfrak{h}$ acts as derivations of both products.
Thus a Lie algebra structure can be defined on $\mathfrak{h} \oplus \mathfrak{m}$ if and only if
identities (LY3) and (LY4) are satisfied.
If $\mathfrak{h} \oplus \mathfrak{m}$ cannot be made into a Lie algebra, then these products
can still be defined, but we will obtain a different class of algebras closely related to LY algebras.

We can refine this approach by considering a simple Lie algebra $\mathfrak{h}$ and a faithful representation
$\rho\colon \mathfrak{h} \to \mathrm{End}_\field(V)$.
We have the embeddings
  \[
  \mathfrak{h}\cong \rho(\mathfrak{h})\subseteq \mathfrak{sl}(V)\subseteq \mathfrak{gl}(V)\cong V\otimes V^*,
  \]
which product reductive pairs $(\mathfrak{sl}(V),\mathfrak{h})$ and the associated LY algebras.
If the $\mathfrak{h}$-module $V$ has an $\mathfrak{h}$-invariant symmetric or skew-symmetric bilinear form
then $\mathfrak{h}$ embeds into either $\mathfrak{so}(V)$ or $\mathfrak{sp}(V)$, and we obtain
reductive pairs $(\mathfrak{so}(V), \mathfrak{h})$ or $(\mathfrak{sp}(V), \mathfrak{h})$.
We will use this approach to understand the structure of the LY algebras associated to the
reductive pairs $(\mathfrak{so}_N(\field), \mathfrak{so}_n(\field))$.

We recall basic facts about orthogonal Lie algebras.
Let $V$ be an $n$-dimensional vector space ($n \ge 2$) over $\field$ with a nondegenerate symmetric bilinear form $\varphi$.
Fixing an orthonormal basis of $V$, we may assume $V = \field^n$ and $\varphi(x,y) = \sum_{i=1}^n x_i y_i$.
If $n = 2$ then $\mathfrak{so}_2(\field) \cong \field$ is the 1-dimensional abelian Lie algebra.
If $n = 3$ then $\mathfrak{so}_3(\field) \cong \mathfrak{sl}_2(\field)$ (that is, $B_1 = A_1$),
the 3-dimensional simple Lie algebra
(recall that $\field$ is algebraically closed).
If $n = 4$ then $\mathfrak{so}_4(\field) \cong \mathfrak{so}_3(\field)\oplus \mathfrak{so}_3(\field)$
(that is, $D_2 = A_1 \cup A_1$) is the direct sum of two simple ideals.
If $n = 5$ then $\mathfrak{so}_5(\field) \cong \mathfrak{sp}_4(\field)$ (that is, $B_2 = C_2$)
can also be defined by a nondegenerate skew-symmetric bilinear form in 4 dimensions.
If $n = 6$ then $\mathfrak{so}_6(\field) \cong \mathfrak{sl}_4(\field)$ (that is, $D_3 = A_3$),
the $4 \times 4$ matrices of trace 0.
For $n \ge 7$, the orthogonal Lie algebras are simple of types $B_k$ ($n=2k+1$) or $D_k$ ($n=2k$).

\begin{table}[h]
\[
\begin{array}{l|ll}
\mathfrak{so}_n(\field) &
\text{Adjoint: dimension $\binom{n}{2}$} &
\text{Natural: dimension $n$}
\\
\midrule
n = 3 &
V(2\lambda_1) &
V(2\lambda_1)
\\[3pt]
n = 4 &
\big( V(2\lambda_1) \otimes \field \big) \oplus \big( \field \otimes V(2\lambda_1') \big) &
V(\lambda_1)\otimes V(\lambda_1')
\\[3pt]
n = 5 &
V(2\lambda_2) &
V(\lambda_1)
\\[3pt]	
n = 6 &
V(\lambda_2+\lambda_3) &
V(\lambda_1)
\\[3pt]
n \ge 7 &
V(\lambda_2)&V(\lambda_1)
\\
\midrule
\end{array}
\]
\caption{Representations of orthogonal Lie algebras}
\label{orthogonalrepresentations}
\end{table}

We write $V(\lambda)$ for the simple $\mathfrak{so}_n(\field)$-module of highest weight $\lambda$;
we use the labelling of the fundamental weights given by \cite{Hu72}.
In the non-simple case $\mathfrak{so}_4(\field) \cong \mathfrak{so}_3(\field) \oplus \mathfrak{so}_3(\field)$,
we use $\lambda$ and $\lambda'$ for the weights for the first and second components.
We record the structure of the adjoint and natural $\mathfrak{so}_n(\field)$-modules in
Table \ref{orthogonalrepresentations}.

\begin{lemma} \label{representationlevel}
For $n \ge 3$, the symmetric square $S^2 \, \field^n$ of the natural $\mathfrak{so}_n(\field)$-module
has dimension $\binom{n+1}{2}$ and decomposes as follows:
\begin{alignat*}{2}
n &= 3\colon &\qquad
S^2 \, \field^3 &= V(4\lambda_1) \oplus V(0)
\\
n &= 4\colon &\qquad
S^2 \, \field^4 &= \big( V(2\lambda_1) \otimes V(2\lambda_1') \big) \oplus \big( V(0) \otimes V(0) \big)
\\
n &\ge 5\colon &\qquad
S^2 \, \field^n &= V(2\lambda_1) \oplus V(0)
\end{alignat*}
\end{lemma}

\begin{proof}
We consider the tensor square of the natural module,
  \[
  \field^n \otimes \field^n = S^2 \, \field^n \oplus \Lambda^2 \, \field^n.
  \]
Since $\field^n$ has an invariant nondegenerate symmetric bilinear form, $V(0)$ is a submodule of $S^2 \, \field^n$.
For $n \ne 4$, the stated decompositions then follow from \cite[Th.~5]{As94} and a dimension count.
For $n = 4$, since $\mathfrak{so}_4(\field) = \mathfrak{so}_3(\field) \oplus \mathfrak{so}_3(\field)$,
the simple $\mathfrak{so}_4(\field)$-modules are tensor products $V \otimes V'$ of
simple $\mathfrak{so}_3(\field)$-modules.
Denoting by $\lambda$ and $\lambda'$ the fundamental weights of the two copies of $\mathfrak{so}_3(\field)$
inside $\mathfrak{so}_4(\field)$, the result follows from the calculation given in detail in Example \ref{m4} below.
\end{proof}

For $n \ge 3$ we write $Q_n$ for the quotient of the $\mathfrak{so}_n(\field)$-module $S^2 \, \field^n$ by
its trivial 1-dimensional submodule; we have $\dim Q_n = N$.

\begin{lemma} \label{lemma2}
For $n \ge 3$, the exterior square $\Lambda^2 \, Q_n$ has dimension $\binom{N}{2}$ and decomposes as follows:
\begin{alignat*}{2}
n &= 3\colon &\qquad
\Lambda^2 \, Q_3 &= V(2\lambda_1)\oplus V(6\lambda_1)
\\
n &= 4\colon &\qquad
\Lambda^2 \, Q_4 &=
\big( V(2\lambda_1) \otimes \field \big) \oplus
\big( V(2\lambda_1)\otimes V(4\lambda_1') \big) \oplus {}
\\
&&&\quad\;
\big( \field \otimes V(2\lambda_1') \big)
\oplus
\big( V(4\lambda_1)\otimes V(2\lambda_1') \big)
\\
n &= 5\colon &\qquad
\Lambda^2 \, Q_5 &= V(2\lambda_2)\oplus V(2\lambda_1+2\lambda_3)
\\
n &= 6\colon &\qquad
\Lambda^2 \, Q_6 &= V(\lambda_2+\lambda_3)\oplus V(2\lambda_1+\lambda_2+\lambda_3)
\\
n &\ge 7\colon &\qquad
\Lambda^2 \, Q_n &= V(\lambda_2)\oplus V(2\lambda_1+\lambda_2)
\end{alignat*}
\end{lemma}

\begin{proof}
Similar to the proof of Lemma \ref{representationlevel}.
\end{proof}

The decompositions in Lemma \ref{lemma2} have a natural multilinear interpretation.
The orthogonal Lie algebra $\mathfrak{so}_n(\field)$ consists of the $n \times n$ skew-symmetric matrices
(with respect to the symmetric bilinear form given by the identity matrix $I_n$).
The space of all $n \times n$ matrices decomposes as the direct sum of symmetric matrices and skew-symmetric matrices:
\begin{equation}
\label{symmskew}
M_n(\field) = \mathfrak{so}_n(\field) \oplus H_n(\field) =
\mathfrak{so}_n(\field) \oplus \big( H_n(\field)_0 \oplus \field I_n \big).
\end{equation}
Using the commutator, $M_n(\field)$ becomes a $\mathbb{Z}_2$-graded Lie algebra
with the simple Lie algebra $\mathfrak{so}_n(\field)$ as the even part and $H_n(\field)$ as the odd part.
Using the anti-commutator, $M_n(\field)$ becomes a $\mathbb{Z}_2$-graded Jordan algebra
with the simple Jordan algebra $H_n(\field)$ as the even part and $\mathfrak{so}_n(\field)$ as the odd part.
The generic trace of the Jordan algebra $H_n(\field)$ is non-degenerate when restricted to the subspace $H_n(\field)_0$
which has dimension $N$; the skew-symmetric $N \times N$ matrices with respect to this trace form
define the orthogonal Lie algebra $\mathfrak{so}(N,\field)$.
The Jordan algebra $H_n(\field)$ is a realization of $S^2\,\field^n$, and $Q_n = H_n(\field)_0$ is then a realization of
the $n \times n$ symmetric matrices with trace 0.
In this way, the LY algebras associated to this sequence of orthogonal reductive pairs are seen to be closely related
to Lie and Jordan algebras.
From the  classification of irreducible LY algebras \cite{BEM11}, we see that for $n \ge 7$ there is only one way to define
binary and ternary products on the $\mathfrak{so}_n(\field)$-module $V(2\lambda_1+\lambda_2)$ to obtain the structure of
an LY algebra.

\subsection{From Jordan algebras to LY algebras}

Recall that the field $\field$ is assumed to be algebraically closed.
We generalize the $\mathbb{Z}_2$-grading of $M_n(\field) = \mathfrak{so}_n(\field) \oplus H_n(\field)$
given by the trace form $\mathrm{tr}(a)$ which is symmetric and non-degenerate.
We consider an $n$-dimensional vector space $V$ with a nondegenerate symmetric bilinear form $\varphi$.
Using $\varphi$ we may identify $V$ and its dual $V^\ast$ by $f \leftrightarrow f^\ast$ where $f^\ast$ is
defined as usual by $\varphi(f^\ast(x),y)=\varphi(x,f(y))$.
We have $\mathrm{End}_\field(V) \cong V^\ast \otimes_\field V$;
interchanging the tensor factors gives $V \otimes_\field V^\ast$,
and using $\varphi$ we can identify this with $\mathrm{End}_\field(V)$.
This gives us an involution on the Lie algebra $\mathfrak{gl}(V)$, which is $\mathrm{End}_\field(V)$ using the
commutator, and from this we obtain the decomposition
\begin{equation}
\label{generalsymmskew}
\mathfrak{gl}(V) = \mathsf{Skew}(V,\varphi) \oplus \mathsf{Sym}(V,\varphi)
\end{equation}
where $\mathsf{Sym}(V,\varphi)$ and $\mathsf{Skew}(V,\varphi)$ are respectively the symmetric and skew-symmetric
endomorphisms defined by $f^*=f$ and $f^* = -f$.
If we fix an orthonormal basis of $V$ with respect to $\varphi$ we obtain the matrix decomposition \eqref{symmskew}.

We consider the following binary and ternary products on $H_n(\field)$:
  \begin{alignat*}{2}
  &\text{Lie bracket} &\qquad [a,b] &= ab-ba
  \\
  &\text{Jordan product} &\qquad a \circ b &= ab+ba
  \\
  &\text{Lie triple product} &\qquad \langle a,b,c \rangle
  &= (b,c,a)_\circ
  \\
  &\text{Jordan triple product} &\qquad \{a,b,c\} &= abc+ cba
  \end{alignat*}
where $(b,c,a)_\circ = (a \circ b) \circ c - a \circ (b \circ c)$ is the Jordan associator
(cyclically permuted).

We recall some basic results, proofs of which may be found in \cite{MH}.
The subspace $H_n(\field)$ is a simple Jordan algebra using the $\circ$ product.
The generic trace of this Jordan algebra is the matrix trace $\mathrm{tr}(a)$,
and $T(a,b) = \mathrm{tr}(ab)$ is a non-degenerate bilinear form, which is associative in the sense that
$T(a\circ b,c)=T(a,b\circ c)$.
The subspace $H_n(\field)$ is a simple Jordan triple system using the brace product $\{\,,\,,\}$.

The linear map $d_{a,b} \in \mathrm{End}(H_n(\field))$ defined by
$d_{a,b}(c) = (b,c,a)_\circ = [L_a,L_b](c)$ where $L_a(b) = a \circ b$
is a derivation of both the Jordan product and the Jordan triple product;
linear combinations of the maps $d_{a,b}$ are called inner derivations.
The full derivation algebras of $H_n(\field)$ with respect to both products coincide and so we may write
$\mathrm{Der}\, H_n(\field)$.
Every derivation is an inner derivation: we have
  \[
  \mathrm{Der}\, H_n(\field)
  =
  \mathrm{span}( \, [L_a,L_b] \mid a,b\in H_n(\field) \, )
  =
  \{ \, \mathrm{ad}_a \mid a \in M_n(\field), \; a^t=-a \, \},
  \]
where $\mathrm{ad}_a(b) = [a,b]$.
In fact, as Lie algebras we have $\mathrm{Der}\, H_n(\field) \cong \mathfrak{so}_n(\field)$,
which is simple for $n \ne 4$.

For any $d \in \mathrm{Der}\, H_n(\field)$ we have $T(d(a),b) + T(a,d(b)) = 0$;
thus the trace form $T$ is invariant under derivations.
The direct sum $H_n(\field) = H_n(\field)_0 \oplus \field I_n$ is orthogonal with respect to $T$.
If $d \in \mathrm{Der} \, H_n(\field)$ then $d(I_n) = 0$ and $d(H_n(\field)) \subseteq H_n(\field)_0$;
hence $H_n(\field)_0$ is a $\mathrm{Der}\, H_n(\field)$-module isomorphic to $V(2\lambda_1)$
as an $\mathfrak{so}_n(\field)$-module.

We have $(b,c,a)_\circ=[[a,b],c]$, and therefore $[[L_a,L_b],L_c]=L_{(b,c,a)_\circ}$
The subspace $H_n(\field)$ is a Lie triple system with respect to the product $\langle \, , \, ,\rangle$.
The decomposition $H_n(\field) = H_n(\field)_0 \oplus \field I_n$ is a direct sum of ideals, and
$H_n(\field)_0$ is a simple Lie triple system.
For any $a,b \in H_n(\field)_0$, the linear map $d_{a,b} = (b,-,a)_\circ$ is a derivation of the Lie triple product
on $H_n(\field)$; the full derivation algebra coincides with $\mathrm{Der}\, H_n(\field)$.
This derivation algebra is a subalgebra of $\mathfrak{so}(H_n(\field)_0,T) \cong \mathfrak{so}_N(\field)$,
the Lie algebra of endomorphisms $d$ of $H_n(\field)_0$ which are orthogonal with respect to $T$,
in the sense that $T(d(a),b) + T(a,d(b)) = 0$ for all $a, b \in H_n(\field)_0$.
The pair
  \[
  \big( \, \mathfrak{so}(H_n(\field)_0,T), \, \mathrm{Der}\, H_n(\field) \, \big)
  =
  \big( \, \mathfrak{so}_N(\field), \, \mathfrak{so}_n(\field) \, \big)
  \]
is a reductive pair and the orthogonal complement $LY_n = (\mathfrak{so}_n(\field))^\perp$
with respect to the Killing form of $\mathfrak{so}_N(\field)$
is the unique $\mathrm{Der}\, H_n(\field)$-invariant complement to $\mathrm{Der}\, H_n(\field)$
in $\mathfrak{so}(H_n(\field)_0, T)$.
Hence, $LY_n$ is the unique standard LY algebra
associated to the reductive pair $(\mathfrak{so}(H_n(\field)_0, T),\mathrm{Der}\, H_n(\field))$.

For $n \ge 3$, the LY algebra $LY_n$ is simple with nontrivial binary and ternary products, and has dimension
  \[
  \dim LY_n
  =
  \binom{N}{2} - \binom{n}{2}
  =
  \binom{ \binom{n+1}{2}-1 }{2} - \binom{n}{2}
  =
  \frac18 (n-1)(n+1)(n-2)(n+4).
  \]
The derivation algebra $\mathrm{Der}\, LY_n$ is isomorphic to $\mathfrak{so}_n(\field)$.
For $n \ne 4$, $LY_n$ is irreducible as a module over its derivation algebra;
the highest weight of this module is given in Lemma \ref{lemma2}.
For $n = 4$, the same lemma shows that
  \begin{equation}
  \label{LY4decomposition}
  LY_4
  =
  U \oplus U'
  =
  \big( V(2\lambda_1)\otimes V(4\lambda_1') \big) \oplus \big( V(4\lambda_1)\otimes V(2\lambda_1') \big).
  \end{equation}
Thus $LY_4$ is a $2$-irreducible module over its derivation algebra and each summand is 15-dimensional.
We now consider some other constructions of $LY_3$ and $LY_4$ which have appeared in the literature.

\subsection{Three different approaches to $LY_3$}

In 1984, Dixmier \cite{Di84} constructed some simple nonassociative algebras using the 3-dimensional Lie algebra
$\mathfrak{sl}_2(\field)$, its irreducible modules, and transvections from classical invariant theory.
In particular, we recall from \cite[\S 6.2]{Di84} the 3-parameter family of simple Lie algebras of type $B_2$
given by a reductive decomposition:
  \begin{equation}\label{m3first construction}
  B_{3,\lambda,\mu,\nu} = V(2\lambda_1) \oplus V(6\lambda_1), \qquad 10\lambda\mu = 9\nu^2.
  \end{equation}
Since $V(2\lambda_1)$ is the adjoint module of $\mathfrak{sl}_2(\field) \cong \mathfrak{so}_3(\field)$
and $B_2 \cong \mathfrak{so}_5(\field)$, we obtain
  \begin{equation}\label{m2first construction}
  \mathfrak{g} = \mathfrak{so}_5(\field)=\mathfrak{sl}_2(\field) \oplus LY_3
  =
  \mathfrak{h} \oplus \mathfrak{m}.
  \end{equation}
This leads to the following explicit construction of $LY_3$.
For each $k \ge 0$, we consider the $(k{+}1)$-dimensional vector space
  \[
  P(k) = \mathrm{span}( x^{k-i} y^i \mid 0 \le i \le k ),
  \]
of homogeneous polynomials in $x, y$ of degree $n$.
We can identify $P(k)$ with the $\mathfrak{sl}_2(\field)$-module $V(k \lambda_1)$ if we define the action
by the differential operators,
  \[
  h = x \frac{\delta}{\delta x} - y\frac{\delta}{\delta y},
  \qquad
  e = x\frac{\delta}{\delta y},
  \qquad
  f = y\frac{\delta}{\delta x}.
  \]
The following equation defines a Lie algebra structure on $B_{3,6,15,10}$:
 \[
 [\ell_1+m_1,\ell_2+m_2] = 6 (\ell_1\ell_2)_1 + 15 (m_1m_2)_5 + 6 (\ell_1m_2)_1 + 6 (m_1\ell_2)_1 + 10(m_1m_2)_3,
 \]
where
  \begin{align*}
  (\ell_1\ell_2)_1
  &=
  \frac14
  \left(
  \frac{\delta \ell_1}{\delta x}\frac{\delta \ell_2}{\delta y}-\frac{\delta \ell_1}{\delta y}\frac{\delta \ell_2}{\delta x}
  \right),
  \\
  (\ell_1m_2)_1
  &=
  \frac{1}{12}
  \left(
  \frac{\delta \ell_1}{\delta x}\frac{\delta m_2}{\delta y}-\frac{\delta \ell_1}{\delta y}\frac{\delta m_2}{\delta x}
  \right),
  \qquad
  (m_1\ell_2)_1
  =
  -(\ell_2 m_1)_1
  \\
  (m_1m_2)_5
  &=
  \sum_{i=0}^5 (-1)^i{5\choose i} \,
  \frac{\delta^{5} m_1}{\delta x^{5-i} \, \delta y^i} \,
  \frac{\delta^{5} m_2}{\delta x^i \, \delta y^{5-i}}
  \\
  (m_1m_2)_3
  &=
  \sum_{i=0}^3 (-1)^i{3\choose i} \,
  \frac{\delta^{3} m_1}{\delta x^{3-i} \, \delta y^i} \,
  \frac{\delta^{3} m_2}{\delta x^i \, \delta y^{3-i}}
  \end{align*}
For $\ell_1 = \ell_2 = 0$ we obtain the Lie bracket of two elements in $LY_3$:
  \[
  [m_1,m_2] = 15 (m_1m_2)_5 + 10(m_1m_2)_3.
  \]
Since $(m_1m_2)_5\in \mathfrak{h}$ and $(m_1m_2)_3\in LY_3$, we apply equations \eqref{standardLY} to obtain
the binary and ternary products on $LY_3$:
  \[
  m_1\cdot m_2 = 10(m_1m_2)_3, \qquad \{m_1,m_2,m_3\} = 90 ((m_1m_2)_5 m_3)_1.
  \]

For a second approach using the representation theory of $\mathfrak{sl}_2(\field)$, see \cite{BD}.
The binary and ternary products of $LY_3$ were obtained using the decomposition
  \[
  \Lambda^2 V(6\lambda_1) = V(2\lambda_1) \oplus V(6\lambda_1) \oplus V(10\lambda_1).
  \]
The binary product is given (up to a scalar) by the projection
$\Lambda^2 V(6\lambda_1) \to V(6\lambda_1)$;
the ternary product is the composition of
$\Lambda^2 V(6\lambda_1) \to V(2\lambda_1) \cong \mathfrak{sl}_2(\field)$
with the action of $\mathfrak{sl}_2(\field)$ on $V(6\lambda_1)$.
Using computer algebra, the polynomial identities in low degree relating these products were studied.

The third approach to $LY_3$ is given by the isomorphism $\mathfrak{so}_5(\field) \cong \mathfrak{sp}_4(\field)$.
Following \cite{BEM11}, we have the reductive decomposition:
  \[
  \mathfrak{sp}_4(\field) = \mathrm{Der}\, \mathcal{T}_\field \oplus LY_3,
  \]
where $\mathcal{T}_\field$ is the symplectic triple system associated to the Jordan algebra of a nondegenerate cubic form
with basepoint given by the norm $n(\alpha) = \alpha^3$ on the base field $\field$.
The triple $\mathcal{T}_\field$ can be identified with the vector space $M_2(\field)$ of $2\times 2$ matrices;
for definitions and formulas for the triple product see \cite[Th.~2.21]{E06}
where a survey of these triples is also included.
Symplectic triple systems are closely related to Freudenthal triple systems, Lie triple systems and 5-graded Lie algebras.

\subsection{The 2-irreducible LY algebra $LY_4$} \label{m4}

Using multilinear algebra we can obtain a convenient model of this LY algebra.
Let $V_1 = V_2 = \field^3$ and let $\varphi_i$ be a nondegenerate symmetric bilinear form on $V_i$.
We define
  \[
  \mathsf{Sym}_0(V_i,\varphi_i)
  =
  \{ \, f \in \mathrm{End}_\field( V_i ) \mid \varphi_i(f(x),y) = \varphi_i(x,f(y)), \; \mathrm{tr}(f) = 0 \, \}.
  \]
On $V_1 \otimes V_2$ we have the nondegenerate symmetric bilinear form
$\varphi = \varphi_1 \otimes \varphi_2$ defined on simple tensors by
$\varphi( x_1 \otimes x_2, y_1 \otimes y_2 ) = \varphi_1( x_1, y_1 ) \varphi_2( x_2, y_2 )$.
Equation (\ref{generalsymmskew}) gives a decomposition of the Lie algebra $\mathfrak{gl}(V_i)$:
  \[
  \mathfrak{gl}(V_i)
  =
  \mathfrak{so}(V_i,\varphi_i) \oplus \mathsf{Sym}_0(V_i,\varphi_i) \oplus \field \, \mathrm{Id}_{V_i}
  \qquad
  (i = 1, 2).
  \]
As vector spaces, $\mathfrak{gl}(V_1) \otimes \mathfrak{gl}(V_2) \cong \mathfrak{gl}(V_1\otimes V_2)$.
On this space we define a Lie bracket by
  \[
  [a\otimes b,f\otimes g] = ab \otimes fg - ba \otimes gf,
  \]
where $ab$ and $fg$ denote compositions of linear maps,
and we obtain the Lie algebra decomposition:
  \[
  \mathfrak{gl}(V_1\otimes V_2)
  =
  \mathfrak{so}(V_1\otimes  V_2,\varphi) \oplus \mathsf{Sym}(V_1\otimes V_2,\varphi).
  \]
From the last two displayed equations we get the reductive decomposition:
  \begin{align*}
  \mathfrak{so}(V_1\otimes  V_2,\varphi)
  &\cong
  \big( \, \mathfrak{so}(V_1,\varphi_1) \otimes \field \, \mathrm{Id}_{V_2} \, \big)
  \oplus
  \big( \, \field \, \mathrm{Id}_{V_1}\otimes \mathfrak{so}(V_2,\varphi_2) \, \big)
  \\
  &\quad
  \oplus
  \big( \, \mathfrak{so}(V_1,\varphi_1) \otimes \mathsf{Sym}_0(V_2,\varphi_2) \, \big)
  \oplus
  \big( \, \mathsf{Sym}_0(V_1,\varphi_1) \otimes \mathfrak{so}(V_2,\varphi_2) \, \big)
  \\
  &\cong
  \big( \, \mathfrak{so}_3(\field)\oplus \mathfrak{so}_3(\field) \, \big)
  \oplus
  LY_4
  \\
  &\cong
  \mathfrak{so}_9(\field),
  \end{align*}
which corresponds to the reductive pair
$(\mathfrak{so}_9(\field),\mathfrak{so}_3(\field)\oplus \mathfrak{so}_3(\field))$.
In this way, we obtain the 2-irreducible simple LY algebra $LY_4 = U \oplus U'$ of equation \eqref{LY4decomposition}.
Using the Lie bracket defined above, we obtain
  \begin{align*}
  [ a \otimes s, b \otimes t ]
  &=
  \tfrac12
  [a,b] \otimes \big( st + ts - \tfrac23 \mathrm{tr}(st) I_3 )
  +
  [a,b] \otimes \tfrac13 \mathrm{tr}(st) I_3
  \\
  &\quad
  +
  \tfrac12
  \big( ab + ba - \tfrac23 \mathrm{tr}(ab) I_3 \big) \otimes [s,t]
  +
  \tfrac13 \mathrm{tr}(ab)I_3 \otimes [s,t],
  \\
  [ s' \otimes a', t' \otimes b' ]
  &=
  \tfrac12 [s',t'] \otimes \big( a'b' + b'a' - \tfrac23 \mathrm{tr}(a'b') I_3 \big)
  +
   [s',t'] \otimes \tfrac13 \mathrm{tr}(a'b') I_3
  \\
  &\quad
  +
  \tfrac12 \big( s't' + t's' - \tfrac23 \mathrm{tr}(s't') I_3 \big) \otimes [a',b']
  +
  \tfrac13 \mathrm{tr}(s't') I_3 \otimes [a',b'],
  \\
  [ a \otimes s, s' \otimes a' ]
  &=
  \tfrac12 ( as' + s'a ) \otimes [s,a']
  +
  \tfrac12 [a,s'] \otimes ( sa' + a's ),
  \end{align*}
where
  \begin{alignat*}{3}
  &
  a,b \in \mathfrak{so}(V_1,\varphi_1),
  &
  \qquad
  &
  s,t\in \mathsf{Sym}_0(V_2,\varphi_2),
  &
  \qquad
  &
  \text{for the summand $U$},
  \\
  &
  s',t'\in \mathsf{Sym}_0(V_1,\varphi_1),
  &
  \qquad
  &
  a',b'\in \mathfrak{so}(V_2,\varphi_2),
  &
  \qquad
  &
  \text{for the summand $U'$}.
  \end{alignat*}
The binary product on $LY_4$ is given by the following equations:
  \begin{align*}
  (a\otimes s) \cdot (b \otimes t)
  &=
  \tfrac12[a,b]\otimes (st+ts-\tfrac{2}{3}\mathrm{tr}(st)I_3)
  +
  \tfrac12(ab+ba-\tfrac{2}{3}\mathrm{tr}(ab)I_3)\otimes [s,t],
  \\
  (a\otimes s) \cdot (s' \otimes a')
  &=
  \tfrac12(as'+s'a)\otimes [s,a']
  +
  \tfrac12 [a,s']\otimes(sa'+a's),
  \\
  (s'\otimes a') \cdot (t' \otimes b')
  &=
  \tfrac12[s',t']\otimes(a'b'+b'a'-\tfrac{2}{3}\mathrm{tr}(a'b')I_3)
  \\
  &\quad
  +
  \tfrac12(s't'+t's'-\tfrac{2}{3}\mathrm{tr}(s't')I_3)\otimes [a',b'].
  \end{align*}
The ternary product is given by the following equations:
  \begin{align*}
  \{a\otimes s,  b \otimes t, c \otimes u\}
  &=
  \tfrac13\mathrm{tr}(st)[[a,b],c]\otimes u
  +
  \tfrac13\mathrm{tr}(ab)c\otimes [[s,t],u],
  \\
  \{a\otimes s,  b \otimes t, s' \otimes a'\}
  &=
  \tfrac13\mathrm{tr}(st)[[a,b],s']\otimes a'
  +
  \tfrac13\mathrm{tr}(ab)s'\otimes [[s,t],a'],
  \\
  \{a\otimes s,  s'\otimes a', b \otimes t\}
  &=
  0,
  \\
  \{a\otimes s,  s'\otimes a', t' \otimes b'\}
  &=
  0,
  \\
  \{s'\otimes a',  t' \otimes b', a \otimes s\}
  &=
  \tfrac13\mathrm{tr}(a'b')[[s',t'],a]\otimes s
  +
  \tfrac13\mathrm{tr}(s't')a\otimes [[a',b'],s],
  \\
  \{s'\otimes a',  t' \otimes b', u' \otimes c'\}
  &=
  \tfrac13\mathrm{tr}(a'b')[[s',t'],u']\otimes c'
  +
  \tfrac13\mathrm{tr}(s't')u'\otimes [[a',b'],c'].
  \end{align*}

\subsection{The reductive decomposition $\mathfrak{so}\left(\binom{n{+}1}{2},\mathbb{R}\right)=\mathfrak{so}(n,\mathbb{R}) \oplus M$}

The Lie algebra $L = \mathfrak{so}(n,\mathbb{R})$ consists of the skew-symmetric
$n \times n$ matrices $A$ with real entries; that is, $A^t = -A$.
We have $\dim L = \binom{n}{2}$; we consider the standard basis for $L$
consisting of the matrices $E_{ij} - E_{ji}$ for $1 \le i < j \le n$.

We write $\mathbb{R}^n$ for the natural $n$-dimensional $L$-module with action given by
$A \cdot X = AX$ for $A \in L$ and $X \in \mathbb{R}^n$, the usual product of a matrix and a column vector.
We identify the symmetric square of this natural module, $S^2 \mathbb{R}^n$, with the vector space
$H = H_n(\mathbb{R})$ of symmetric $n \times n$ matrices.
This becomes an $L$-module with the action given by the Lie bracket of matrices
$$ A \cdot S = [A,S] = AS - SA \text{ for } A \in L, S \in H. $$
(This is the symmetrization of the left action of $L$ on $H$, since $AS + S^t\!A^t = AS - SA$.)

We have $\mathfrak{gl}(n,\mathbb{R}) = L \oplus H$ as vector spaces.
We write $H^0 \subset H$ for the subspace of symmetric matrices with trace zero,
and so $\mathfrak{sl}(n,\mathbb{R}) = L \oplus H^0$.
This is a reductive decomposition of $\mathfrak{sl}(n,\mathbb{R})$: $L$ is a Lie subalgebra, $[L,L] \subseteq L$,
and $H^0$ is an $L$-module for the restriction of the adjoint representation, $[L,H^0] \subseteq H^0$.
We have that $\dim H = \binom{n{+}1}{2}$ and so $\dim H^0 = \binom{n{+}1}{2} - 1$.
We denote $N = \binom{n{+}1}{2}$.

Thus there is an embedding of $\mathfrak{so}(n)$ into $\mathfrak{so}( N )$.
We require an orthonormal basis of $H$ so that the $N \times N$ matrices representing
the action of $L \subset \mathfrak{so}( N )$ on $H$ are skew-symmetric.
Since $\mathfrak{sl}(n,\mathbb{R})$ is a simple Lie algebra, its Killing form is a (nonzero) scalar multiple of the trace
form on $n \times n$ matrices: $Tr(A,B) = \mathrm{trace}(AB)$.
We choose as standard basis for $H$ the diagonal matrix units $E_{ii}$ for $1 \le i \le n$ together with the matrices
$\frac{1}{\sqrt{2}}( E_{ij} + E_{ij} )$ for $1 \le i < j \le n$.
The scalars $\frac{1}{\sqrt{2}}$ are required to make this basis orthonormal with respect to the trace form.
So the action of $L \subset \mathfrak{so}( N )$ on $H$ is given as follows:
  \begin{align*}
  ( E_{ij} - E_{ji} ) \cdot E_{kk}
  &=
  \delta_{jk} ( E_{ik} + E_{ki} ) - \delta_{ik} ( E_{jk} + E_{kj} ),
  \\[2pt]
  ( E_{ij} - E_{ji} ) \cdot \tfrac{1}{\sqrt{2}}( E_{k\ell} + E_{\ell k} )
  &=
  \delta_{jk} \tfrac{1}{\sqrt{2}} ( E_{i\ell} + E_{\ell i} )
  + \delta_{j\ell} \tfrac{1}{\sqrt{2}} ( E_{ik} + E_{ki} )
  \\
  &  \quad
  - \delta_{ik} \tfrac{1}{\sqrt{2}} ( E_{j\ell} + E_{\ell j} )
  - \delta_{i\ell} \tfrac{1}{\sqrt{2}} ( E_{jk} + E_{kj} ).
  \end{align*}

The 1-dimensional subspace of $H$ spanned by the identity matrix $I$ is annihilated by the action of $\mathfrak{so}(n)$
so it also has to be annihilated by the action of its image in $\mathfrak{so}(N)$.
The subalgebra $U \subset \mathfrak{so}(N)$ which annihilates $I$ is isomorphic to $\mathfrak{so}(N{-}1)$, so we have
$L = \mathfrak{so}(n) \subset \mathfrak{so}(N{-}1) = U$.
We have a reductive decomposition $U = L \oplus M$, where $M = L^\perp$,
the orthogonal complement of $L$ with respect to the trace form on $U$.
So we have $[ L, M ] \subseteq M$. We provide $M$ with the structure of a Lie-Yamaguti algebra
using the Lie bracket $[A,B]_U$ in $U$ together with the projections $p_L$ and $p_M$ onto $L$ and $M$.
The bilinear operation $[A,B]$ and the trilinear operation $\langle A, B, C \rangle$ are defined by
  \begin{equation} \label{projections}
  [ A, B ] = p_M( \, [ A, B ]_U ),
  \qquad
  \langle A, B, C \rangle = p_M( \, [ \, p_L( [ A, B ]_U ), \, C \, ]_U ).
  \end{equation}
The dimension of this Lie-Yamaguti algebra is
  \[
  \dim M
  =
  \binom{N{-}1}{2} - \binom{n}{2}
  =
  \binom{\binom{n{+}1}{2}{-}1}{2} - \binom{n}{2}
  =
  \frac18 (n{-}2)(n{-}1)(n{+}1)(n{+}4).
  \]
For $n = 3, \dots, 9$ we obtain $\dim M = 7$, 30, 81, 175, 330, 567, 910.

\begin{lemma}
  For $n \ge 3$ and $n \ne 4$ the $L$-module $M$ is irreducible.
\end{lemma}

%%%%%%%%%%%%%%%%%%%%%%%%%%%%%%%%%%%%%%%%%%%%%%%%%%%%%%%%%%%%%%%%%%%%%%%%%%%%%%%%%%%%%%%%%%%%%%%%%%%%%%%%%%%%
%%%%%%%%%%%%%%%%%%%%%%%%%%%%%%%%%%%%%%%%%%%%%%%%%%%%%%%%%%%%%%%%%%%%%%%%%%%%%%%%%%%%%%%%%%%%%%%%%%%%%%%%%%%%
%%%%%%%%%%%%%%%%%%%%%%%%%%%%%%%%%%%%%%%%%%%%%%%%%%%%%%%%%%%%%%%%%%%%%%%%%%%%%%%%%%%%%%%%%%%%%%%%%%%%%%%%%%%%

\section{Nonassociative polynomial identities} \label{sectionnonassoc}

Before explaining our computational methods to find the polynomial identities satisfied by the LY algebra $M$
described below we need to introduce some of the concepts involved.
(For further details, see \cite{BD}.)

\subsection{Nonassociative polynomials}

Let $\Omega$ be a set of multilinear operations defined on a vector space $A$.
A \textit{polynomial identity of degree $d$ for the algebra $A$} is an element $f \in F\{ \Omega; X \}$
of the free multioperator algebra with operations $\Omega$ and generators $X$ such that
$f(a_1, \dots, a_d) = 0$ for all $a_1, \dots, a_d \in A$. We denote it by $f \equiv 0$.
Over a field $\field$ of characteristic 0, every polynomial identity is equivalent to a finite set of multilinear identities \cite{ZSSS}.
Given a set of operations $\Omega$, the set of \textit{association types of degree $d$ for $\Omega$} is the set of
all possible combinations of operations from $\Omega$ involving $d$ variables.
In our case, $\Omega = \{ [-,-], (-,-,-) \}$ since LY algebras have a binary and a ternary operation;
in the rest of the paper we use this notation, $[a,b]$ instead of $a \cdot b$, and $(a,b,c)$ instead of $\langle a,b,c \rangle$.
Thus we can consider the association types for $[-,-]$, the association types for $(-,-,-)$ and
the association types for $[-,-]$ and $(-,-,-)$.
Owing to the skew-symmetry of the operations in LY algebras, there are association types which are linearly dependent,
like $ [[-,-],-] $ and $ [-,[-,-]] $.
The next examples give the independent association types in low degrees for $[-,-]$ and $(-,-,-)$ both separately and mixed.

\begin{example} \label{btypes}
Binary types: association types in degrees $d \le 7$ for a skew-symmetric binary operation (the commas within monomials are omitted):
    \begin{align*}
    \!\!\!\!\!\!\!
    d &= 1\colon -
    \quad
    d = 2\colon [ - - ]
    \quad
    d = 3\colon
    [ [ - - ] - ]
    \quad
    d = 4\colon
    [ [ [ - - ] - ] - ], \;  [ [ - - ] [ - - ] ]
    \\[-18pt]
    \end{align*}
    \begin{alignat*}{4}
    d &= 5\colon &
    &[ [ [ [ - - ] - ] - ] - ] &\qquad &[ [ [ - - ] [ - - ] ] - ] &\qquad &[ [ [ - - ] - ] [ - - ] ]
    \\[3pt]
    d &= 6\colon &
    &[ [ [ [ - - ] - ] - ] - ] - ] &\qquad &[ [ [ [ - - ] [ - - ] ] - ] - ] &\qquad &[ [ [ [ - - ] - ] [ - - ] ] - ]
    \\
    &&&[ [ [ [ - - ] - ] - ] [ - - ] ] &\qquad &[ [ [ - - ] [ - - ] ] [ - - ] ] &\qquad &[ [ [ - - ] - ] [ [ - - ] - ] ]
    \\[3pt]
    d &= 7\colon &
    &[ [ [ [ [ [ - - ] - ] - ] - ] - ] - ] &\qquad & [ [ [ [ [ - - ] [ - - ] ] - ] - ] - ] &\qquad & [ [ [ [ [ - - ] - ] [ - - ] ] - ] - ]
    \\
    &&&[ [ [ [ [ - - ] - ] - ] [ - - ] ] - ] &\qquad & [ [ [ [ - - ] [ - - ] ] [ - - ] ] - ] &\qquad & [ [ [ [ - - ] - ] [ [ - - ] - ] ] - ]
    \\
    &&&[ [ [ [ [ - - ] - ] - ] - ] [ - - ] ] &\qquad & [ [ [ [ - - ] [ - - ] ] - ] [ - - ] ] &\qquad & [ [ [ [ - - ] - ] [ - - ] ] [ - - ] ]
    \\
    &&&[ [ [ [ - - ] - ] - ] [ [ - - ] - ] ] &\qquad & [ [ [ - - ] [ - - ] ] [ [ - - ] - ] ]
   \end{alignat*}
\end{example}

\begin{example} \label{ttypes}
Ternary types: association types in degrees $d \le 7$ for a ternary operation skew-symmetric in its first two arguments
(commas omitted):
  \begin{align*}
  d = 1\colon - \qquad
  d = 3\colon ( - - - ) \qquad
  d = 5\colon ( ( - - - ) - - ), \quad ( - - ( - - - ) )
  \\[-18pt]
  \end{align*}
  \begin{alignat*}{4}
  d &= 7\colon &
  &( ( ( - - - ) - - ) - - ) &\qquad & ( ( - - ( - - - ) ) - - ) &\qquad & ( ( - - - ) ( - - - ) - )
  \\
  &&&( ( - - - ) - ( - - - ) ) &\qquad & ( - - ( ( - - - ) - - ) ) &\qquad & ( - - ( - - ( - - - ) ) )
  \end{alignat*}
\end{example}

\begin{example} \label{bttypes}
Mixed types: association types in degrees $d \le 5$ including those involving both operations;
the types from examples \ref{btypes} and \ref{ttypes} also appear:
  \[
  d = 1\colon -
  \qquad
  d = 2\colon [ - - ]
  \qquad
  d = 3\colon [ [ - - ] - ], \quad ( - - - )
  \]
  \[
  d = 4\colon
  [ [ [ - - ] - ] - ] \qquad  [ ( - - - ) - ] \qquad   [ [ - - ] [ - - ] ] \qquad ( [ - - ] - - ) \qquad   ( - - [ - - ] )
  \]
  \begin{alignat*}{5}
  d &= 5\colon &
  &[[[[[ - - ] - ] - ] - ] &\quad  &[[( - - - ) - ] - ] &\quad  &[[[ - - ] [ - - ]] - ] &\quad   &[([ - - ] - - ) - ]
  \\
  &&&[( - - [ - - ]) - ] &\quad   &[[[ - - ] - ] [ - - ]] &\quad   &[( - - - ) [ - - ]] &\quad   &([[ - - ] - ] - - )
  \\
  &&&(( - - - ) - - ) &\quad   &([ - - ] [ - - ] - ) &\quad   &([ - - ] - [ - - ]) &\quad  & ( - - [[ - - ] - ])
  \\
  &&&( - - ( - - - ))
  \end{alignat*}
For degrees $d = 6, 7$ there are respectively 38 and 113 independent types.
\end{example}

Given a set of operations $\Omega$, the (multilinear) \textit{monomials of degree} $d$ are obtained by applying
all possible permutations of the $d$ variables to the association types of degree $d$.
Since the LY operations have skew-symmetries, there are monomials which are linearly dependent,
such as $[[x_1,x_2],x_3]$ and $[[x_2,x_1],x_3]$, or $(x_1,x_2,[x_3,x_4])$ and $(x_2,x_1,[x_4,x_3])$.
Considering only independent types and monomials reduces the memory needed for the computations.
We consider the equivalence classes of monomials determined by the skew-symmetries;
each equivalence class is represented by its normal form as defined in \cite{BD}.

A multilinear polynomial $f \in F\{ \Omega; X \}$ of degree $d$ has consequences in higher degrees,
which are obtained using the operations $\omega \in \Omega$.
To lift $f$ using an operation $\omega$ of arity $r$ we need to add $r-1$ new variables
so the degree of the corresponding \emph{liftings} is $d+r-1$.
We obtain $d$ liftings by replacing $x_i$ by $\omega(x_i,x_{d+1},\cdots,x_{d+r-1})$ and another $r$ liftings by making $f(x_1,\cdots,x_d)$
the $j$-th argument of $\omega$.
To lift a given polynomial to a higher degree $D$ we need to consider all the possible combinations of liftings by operations from $\omega$
which increase degree $d$ to degree $D$.
The liftings of $f$ to degree $D$ form a (not necessarily independent) set of generators for the $S_D$-module of all consequences of $f$
in degree $D$.

For example, in the case $\Omega = \{ [-,-], (-,-,-) \}$, to lift a polynomial from degree 2 to degree 4
we can use the binary operation twice or the ternary operation once.
Owing to the skew-symmetries of the operations in a LY algebra some of the liftings are redundant.

For each degree $d$ the sets of binary, ternary and mixed monomials in normal form are bases of the corresponding
spaces of polynomials. These spaces are $S_d$-modules, the action given by permuting the subscripts of the variables.
We are interested in finding module generators for the subspaces of multilinear identities of the LY algebra $M$ in each degree
to determine whether they can be deduced from the defining identities for LY algebras.

\subsection{Computational methods}

For each degree $d$, we perform the following steps:
\begin{enumerate}
    \item Generate the bases of binary, ternary and mixed monomials in normal form.
    \item Find the $S_d$-modules of all identities of degree $d$ satisfied by $M$, considering separately binary, ternary and mixed monomials.
    \item Obtain the corresponding $S_d$-submodules of lifted identities for $M$ from degrees $< d$.
    \item Compute the quotient modules (all identities modulo lifted identities) and (if they are not trivial) find their minimal sets of generators;
    these generators are the new identities for $M$ in degree $d$.
    \item Compare the generators for the quotient module of new identities with the LY identities of degree $d$
    (the defining identities for LY algebras).
\end{enumerate}
We use two algorithms: ``fill and reduce'' determines a basis for the vector space of
(binary, ternary or mixed) polynomial identities in degree $d$ satisfied by $M$;
``module generators'' applies the representation theory of the symmetric group $S_d$ to extract a subset of the basis
which generates the space of identities as an $S_d$-module.
A modification of ``fill and reduce'' also takes into account the liftings to degree $d$ of the known identities in lower degrees.

\subsubsection{Fill and reduce}
Let $m$ be the dimension of $M$, and let $q$ be the number of normal (binary, ternary or mixed) monomials of degree $d$.
We construct a zero matrix $E$ of size $(q+m) \times q$ whose columns are labeled by the normal monomials.
We divide it into an upper block of size $q \times q$ and a lower block of size $m \times q$ whose rows are labeled
by the elements of the chosen basis of $M$.
We perform the following steps until the rank of $E$ has stabilized; that is, the rank has not increased for some fixed number $s$ of iterations:
\begin{enumerate}
    \item Generate $d$ pseudorandom column vectors $a_1, \cdots, a_d$ of dimension $m$, representing elements of $M$
    with respect to the given basis.
    \item For each $j = 1, \cdots, q$, evaluate the $j$-th monomial on the elements $a_1, \cdots, a_d$ using the LY operations
            $[-,-]$ and $(-,-,-)$ and store the coordinates of the result in rows $q{+}1$ to $q{+}m$ of column $j$.
            After this, each row of the lower block of $E$ represents a linear relation which must be satisfied by the coefficients of
            any identity satisfied by $M$.
    \item Compute the row canonical form of $E$; since it has size $(q+m) \times q$, its rank must be $\le q$,
            so the lower block of $E$ becomes zero again.
\end{enumerate}
After the rank has stabilized, the nullspace of $E$ consists of the coefficient vectors of the linear dependence relations
between the monomials which are satisfied by many pseudorandom choices of elements of $M$.
We extract a basis of this nullspace, and compute the row canonical form of the matrix whose rows are these basis vectors;
the rows of the reduced matrix represent nonassociative polynomials which are ``probably'' polynomial
identities satisfied by $M$. They still need to be proven directly, or at least checked by another independent computation.

The construction of LY algebras by reductive pairs is defined over a field of characteristic zero.
To optimize computer time and memory we often use modular instead of rational arithmetic, especially in higher degrees.
When using modular arithmetic the rank of the matrix $E$ could be smaller than it would be using
rational arithmetic, and thus the nullspace will be too big.
To avoid this we compute a basis of the nullspace using modular arithmetic and reconstruct
the most probable corresponding integral vectors.  We then check each of these identities
using rational arithmetic by generating pseudorandom integral elements of $M$ and evaluating
the corresponding polynomial identity.
If we obtain zero for a large number of choices, we have obtained further confirmation of the hypothetical identity.
For further information, see \cite[Lemma 8]{BP2009}.

To reduce the number of identities we need to check, we first extract a set of module generators from the linear basis of the nullspace.

\subsubsection{Module generators}
Let $I_1, \dots, I_\ell$ be a linear basis for the $S_d$-module of identities in degree $d$ satisfied by $M$;
that is, a basis for a certain subspace of the $q$-dimensional vector space of nonassociative polynomials in degree $d$.
We construct a zero matrix $G$ of size $( q + d !) \times q$ whose columns are labeled by the normal monomials.
We divide it into a $q \times q$ upper block and a $d\;\!! \times q$ lower block whose rows are labeled
by the elements of the symmetric group $S_d$.
We set $\mathrm{oldrank} \leftarrow 0$ and then perform the following steps for $k = 1, \dots, \ell$:
\begin{enumerate}
    \item Set $i \leftarrow 0$.
    \item For each permutation $\pi$ in the symmetric group $S_d$ do the following:
            \begin{enumerate}
                \item Set $i \leftarrow i+1$.
                \item For each $j = 1, \dots, q$ do the following:
                        \begin{itemize}
                            \item Let $c_j$ be the coefficient in $I_k$ corresponding to monomial $m_j$.
                            \item If $c_j \ne 0$ then apply $\pi$ to $m_j$ obtaining $\pi m_j$ and replace $\pi m_j$ by $\pm m_{\overline{j}}$
                            where $m_{\overline{j}}$ is the normal form (the representative of the equivalence class of $\pi m_j$).
                            \item Store the corresponding coefficient $\pm c_j$ in row $q+i$ and column $\overline{j}$ of $G$.
                        \end{itemize}
            \end{enumerate}
    \item Compute the row canonical form of $G$; the lower block of $G$ is again zero.
    \item Set $\mathrm{newrank} \leftarrow \mathrm{rank}(G)$.
    \item If $\mathrm{oldrank} < \mathrm{newrank}$ then:
            \begin{enumerate}
                \item Record $I_k$ as a new module generator.
                \item Set $\mathrm{oldrank} \leftarrow \mathrm{newrank}$.
            \end{enumerate}
\end{enumerate}
If we already know a set of generators for the lifted identities in degree $d$ then we apply the module generators algorithm to them.
At the end of this computation the row space of $G$ contains a basis for the subspace of identities in degree $d$
which are consequences of identities of lower degree.
We then apply the module generators algorithm to the linear basis produced by the fill and reduce algorithm: the
canonical basis of the nullspace of $E$ which gives a basis for the subspace of all identities in degree $d$.
In this way we obtain a set of module generators for the space of all identities in degree $d$ modulo
the space of known identities in degree $d$; that is, a set of module generators for the new identities in degree $d$.

%%%%%%%%%%%%%%%%%%%%%%%%%%%%%%%%%%%%%%%%%%%%%%%%%%%%%%%%%%%%%%%%%%%%%%%%%%%%%%%%%%%%%%%%%%%%%%%%%%%%%%%%%%%%
%%%%%%%%%%%%%%%%%%%%%%%%%%%%%%%%%%%%%%%%%%%%%%%%%%%%%%%%%%%%%%%%%%%%%%%%%%%%%%%%%%%%%%%%%%%%%%%%%%%%%%%%%%%%
%%%%%%%%%%%%%%%%%%%%%%%%%%%%%%%%%%%%%%%%%%%%%%%%%%%%%%%%%%%%%%%%%%%%%%%%%%%%%%%%%%%%%%%%%%%%%%%%%%%%%%%%%%%%

\section{Computational results for the LY algebra $LY_4$} \label{sectionresults}

For the case $\mathfrak{so}(3,\mathbb{R}) \subset \mathfrak{so}(5,\mathbb{R})$, the bilinear operation gives the 7-dimensional LY algebra $LY_3$
the structure of a simple non-Lie Malcev algebra; the polynomial identities for this LY algebra have been studied in \cite{BD}.

In the rest of this section, we consider the 30-dimensional LY algebra $LY_4$ obtained from the embedding
$\mathfrak{so}(4,\mathbb{R}) \subset \mathfrak{so}(9,\mathbb{R})$.
We note that identities for the trilinear operation occur only in odd degrees.
We chose to perform $s = 10$ iterations of the fill-and-reduce algorithm after the rank stabilizes.

\subsection{Identities for $LY_4$ in degree 3}

We note that identities (LY1) and (LY2) in the definition of LY algebras have already been taken into account
in our definition of normal monomials.

\subsubsection{Binary identities}

There is only one association type for $[-,-]$ in degree 3, $[[-,-],-]$, and three normal monomials, $[[a,b],c]$, $[[a,c],b]$, $[[b,c],a]$.
Thus the matrix for the fill-and-reduce algorithm has size $33 \times 3$, and it reaches full rank after one iteration.
Hence the binary product by itself satisfies no identities of degree 3 other than the consequences of anticommutativity.

\subsubsection{Ternary identities}

There is only one association type for $(-,-,-)$ in degree 3, $(-,-,-)$, and three normal monomials, $(a,b,c)$, $(a,c,b)$, $(b,c,a)$.
Thus the matrix for the fill-and-reduce algorithm has size $33 \times 3$, and it reaches full rank after one iteration.
Hence the ternary product by itself satisfies no identities of degree 3 other than the consequences of skew-symmetry in the first two arguments.

\subsubsection{Mixed identities}

Using the normal monomials for both binary and ternary products gives an ordered basis for the 6-dimensional space of
multilinear LY polynomials in degree 3; these monomials label the columns of the $36 \times 6$ matrix for the fill-and-reduce algorithm:
  \[
  [[a,b],c], \quad [[a,c],b], \quad [[b,c],a], \quad (a,b,c), \quad (a,c,b), \quad (b,c,a).
  \]
The rank reaches 5 after one iteration and remains unchanged for 10 iterations, so there is an identity of degree 3
relating the binary and the ternary products.
At this point the row canonical form of the matrix is
  \[
  \left[
  \begin{array}{rrrrrr}
  1 & 0 & 0 & 0 & 0 & 102 \\
  0 & 1 & 0 & 0 & 0 &   1 \\
  0 & 0 & 1 & 0 & 0 & 102 \\
  0 & 0 & 0 & 1 & 0 & 102 \\
  0 & 0 & 0 & 0 & 1 &   1
  \end{array}
  \right]
  \]
A basis for the nullspace is the vector $[1,102,1,1,102,1]$; using representatives from $-51$ to 51 modulo 103 we obtain
the vector $[1,-1,1,1,-1,1]$, which corresponds to identity (LY3) in the definition of LY algebras:
  \[
  [[a,b],c] - [[a,c],b] + [[b,c],a] + (a,b,c) - (a,c,b) + (b,c,a) \equiv 0,
  \]

\subsubsection{Conclusion}

There are no identities for $LY_4$ in degree 3 which are not consequences of the defining identities of LY algebras.

\subsection{Identities for $LY_4$ in degree 4}

Using the module generators algorithm, we find that the $S_4$-module of consequences in degree 4 of the identity $LY3(a,b,c) \equiv 0$
has dimension 10 and is generated by the identities $LY3([a,d],b,c) \equiv 0$ and $[ LY3(a,b,c), d ] \equiv 0$.

\subsubsection{Binary identities}

The two association types for $[-,-]$ in degree 4, namely $[[[-, -], -], -]$ and $[[-, -], [-, -]]$, have respectively 12 and 3 normal monomials.
The matrix for the fill-and-reduce algorithm has size $45 \times 15$, and it reaches full rank after one iteration.
Hence the binary product by itself satisfies no identities of degree 3 other than the consequences of anticommutativity.

\subsubsection{Mixed identities}

The five mixed association types in degree 4, namely
  \[
  [[[-,-],-],-], \quad
  [(-,-,-),-], \quad
  [[-,-],[-,-]], \quad
  ([-,-],-,-), \quad
  (-,-,[-,-]),
  \]
have respectively 12, 12, 3, 12, 6 normal monomials.
The matrix for the fill-and-reduce algorithm has size $75 \times 45$.
The rank reaches 26 after one iteration and remains unchanged for 10 iterations, so there is a 19-dimensional space of identities of degree 4
relating the binary and the ternary products.
We obtain a basis for the nullspace by setting the free variables equal to the standard basis vectors and solving for the leading variables,
and then compute the row canonical form of the nullspace basis.
The coefficients of the canonical basis vectors are 0, 1 and $102 \equiv -1$.
Interpreting the coefficients 0, $\pm 1$ as integers, we sort the basis vectors by increasing square length (from 3 to 18).
We use the module generators algorithm to compare these new identities with the submodule generated by the lifted identities,
and find that there are two new identities which cause the rank to increase from 10 to 19.
These two identities are (LY4) and (LY5) in the definition of LY algebra.

\subsubsection{Conclusion}

There are no identities for $LY_4$ in degree 4 which are not consequences of the defining identities of LY algebras.

\subsection{Identities for $LY_4$ in degree 5}

Using the module generators algorithm, we find the dimension of, and generators for, the $S_5$-module of consequences in degree 5 of
$LY3(a,b,c) \equiv 0$, $LY4(a,b,c,d) \equiv 0$, $LY5(a,b,c,d) \equiv 0$.
The module has dimension 280 and 11 generators:
  \begin{alignat*}{3}
  &LY3([[a,e],d],b,c) \equiv 0, &\qquad
  &LY3([a,d],[b,e],c) \equiv 0, &\qquad
  &[LY3([a,d],b,c),e] \equiv 0,
  \\
  &[[LY3(a,b,c),d],e] \equiv 0, &\qquad
  &LY3((a,d,e),b,c) \equiv 0, &\qquad
  &(LY3(a,b,c),d,e) \equiv 0,
  \\
  &LY4([a,e],b,c,d) \equiv 0, &\qquad
  &LY4(a,b,[c,e],d) \equiv 0, &\qquad
  &[LY4(a,b,c,d),e] \equiv 0,
  \\
  &LY5([a,e],b,c,d) \equiv 0, &\qquad
  &[LY5(a,b,c,d),e] \equiv 0.
  \end{alignat*}

\subsubsection{Binary identities}

There are three association types for $[-,-]$ in degree 5 and 105 normal monomials.
The matrix for the fill-and-reduce algorithm has size $135 \times 105$.
After 4 iterations it reaches full rank, so the binary product by itself satisfies no identities in degree 5 other than the consequences of anticommutativity.

\subsubsection{Ternary identities}

There are two association types for $(-,-,-)$ in degree 5 and 90 normal monomials.
The matrix for the fill-and-reduce algorithm has size $120 \times 90$.
After two iterations, the rank reaches 60, and does not increase for another 10 iterations.
Hence the ternary product satisfies a 30-dimensional space of identities in degree 5.
We use the module generators algorithm to compare these identities with the submodule generated by the lifted identities,
and find that there is one new identity which causes the rank to increase from 280 to 296;
this identity is (LY6).
(The intersection of the space of lifted identities with the space of ternary identities has dimension $280 + 30 - 296 = 14$.)

\subsubsection{Mixed identities}

There are 13 association types involving both operations in degree 5, and 510 normal monomials.
Thus the matrix for the fill-and-reduce algorithm has size $540 \times 510$.
After 8 iterations, the rank reaches 214, and does not increase for another 10 iterations.
Hence there is a 296-dimensional space of identities in degree 5 involving both operations.
But this space contains the space of the same dimension generated by the lifted identities and (LY6),
so there are no more new identities in degree 5.

\subsubsection{Conclusion}

Thus there are no identities of degree 5 for $LY_4$ except those which are consequences of the
defining identities for LY algebras.

\subsection{Identities for $LY_4$ in degree 6}

We extended these computations to degree 6 and found that every identity satisfied by $LY_4$
is a consequence of the defining identities for LY algebras.
The space of all multilinear LY polynomials has dimension 7245, and the liftings of the defining identities for LY algebras
generates a submodule of dimension 5151.
The fill-and-reduce algorithm (for the mixed identities) stabilizes at rank 2094, indicating a nullspace of dimension 5151.

\medskip

We summarize the computations of this section in the following theorem.

\begin{theorem}
Every multilinear polynomial identity of degree $\le 6$ satisfied by the LY algebra $LY_4$ is a consequence of
the defining identities for LY algebras.
\end{theorem}

%%%%%%%%%%%%%%%%%%%%%%%%%%%%%%%%%%%%%%%%%%%%%%%%%%%%%%%%%%%%%%%%%%%%%%%%%%%%%%%%%%%%%%%%%%%%%%%%%%%%%%%%%%%%
%%%%%%%%%%%%%%%%%%%%%%%%%%%%%%%%%%%%%%%%%%%%%%%%%%%%%%%%%%%%%%%%%%%%%%%%%%%%%%%%%%%%%%%%%%%%%%%%%%%%%%%%%%%%
%%%%%%%%%%%%%%%%%%%%%%%%%%%%%%%%%%%%%%%%%%%%%%%%%%%%%%%%%%%%%%%%%%%%%%%%%%%%%%%%%%%%%%%%%%%%%%%%%%%%%%%%%%%%

\section{The Jordan triple product as trilinear operation} \label{sectionjordan}

Grishkov and Shestakov \cite{GS1} define a \emph{Lie-Jordan algebra} to be a vector space with a bilinear product $[-,-]$
and a trilinear product $\{-,-,-\}$ satisfying the following multilinear identities for all $x, y, z, t, u$:
  \begin{align*}
  &[x,y] \equiv -[y,x], \qquad \{x,y,z\} \equiv \{z,y,x\}, \qquad [[x,y],z] \equiv \{x,y,z\} - \{y,x,z\}, \\
  &[\{x,y,z\},t] \equiv \{[x,t],y,z\} + \{x,[y,t],z\} + \{x,y,[z,t]\}, \\
  &\{\{x,y,z\},t,u\} \equiv \{\{x,t,u\},y,z\} - \{x,\{y,u,t\},z\} + \{x,y,\{z,t,u\}\}.
  \end{align*}
Thus a Lie-Jordan algebra is a Lie algebra with respect to the bilinear product and a Jordan triple system with respect to the trilinear product.
A Lie-Jordan algebra is called \emph{special} if it is isomorphic to a subspace of an associative algebra with the standard Lie bracket and
Jordan triple product:
  \[
  [x,y] = xy-yx, \qquad \{x,y,z\} = xyz+zyx.
  \]
From the construction of universal enveloping algebras of Lie-Jordan algebras \cite{GS1}, it follows that
every Lie-Jordan algebra over a field of characteristic not 2 is special.
The same authors \cite{GS2} show that a Lie algebra $L$ with bilinear product $[-,-]$ is isomorphic to a Lie algebra of skew-symmetric elements of
an associative algebra with involution if and only if $L$ admits an additional trilinear operation $\{-,-,-\}$ such that the resulting structure is a
Lie-Jordan algebra.

Motivated by these considerations, we recall that
the Lie algebra $\mathfrak{so}(n,\field)$ consists of the $n \times n$ matrices which are skew-symmetric with respect to the transpose involution,
and so this subspace is also closed under the Jordan triple product.
As explained in Section \ref{sectionsymmetric}, the underlying vector space of the Lie-Yamaguti algebra $LY_n$ is the orthogonal complement of
$\mathfrak{so}(n,\field)$ with respect to its natural embedding into $\mathfrak{so}(N,\field)$.
So we can retain the original bilinear operation on $LY_n$ but replace the trilinear operation by the projection of the Jordan triple product.
That is, for
  \[
  U = L \oplus M, \qquad U = \mathfrak{so}(N{-}1), \qquad L = \mathfrak{so}(n), \qquad M = LY_n,
  \]
we use the associative matrix product to replace equations \eqref{projections} for $A, B, C \in M$ by
  \begin{equation} \label{projections2}
  [ A, B ] = p_M( \, AB-BA \, ),
  \qquad
  \langle A, B, C \rangle = p_M( \, ABC+CBA \, ).
  \end{equation}
We call these new structures \emph{Lie-Jordan-Yamaguti algebras} or LJY algebras for short.
We write $LJY_n$ for the subspace $LY_n$ with the same bilinear operation as $LY_n$ but the new trilinear operation.

LJY algebras generalize Lie algebras and Jordan triple systems in a way somewhat analogous to the way that
LY algebras generalize Lie algebras and Lie triple systems.
However, an LJY algebra with zero bilinear product is a triple system with a completely symmetric trilinear operation satisfying the Jordan triple derivation
identity, which provides a Jordan analogue of Filippov's 3-Lie algebras \cite{F}.
On the other hand, an LJY algebra with zero trilinear product is a 2-step nilpotent anticommutative (hence Lie) algebra.

In this section we study the polynomial identities satisfied by $LJY_3$ and $LJY_4$.

\subsection{The LJY algebra $LJY_3$}

In this case the underlying vector space $LJY_3$ with the anticommutative bilinear product defines the 7-dimensional simple non-Lie Malcev algebra;
however, the symmetric trilinear product comes from the projection of the Jordan triple product obtained from the orthogonal complement of
the embedding of $\mathfrak{so}(3)$ into $\mathfrak{so}(5)$.
The computational methods are the same as before, but evaluating the products in the algebra is much easier since the tables of structure constants
are much smaller.

We recall the the multilinear form of the Malcev identity,
  \[
  [[a,c],[b,d]] \equiv [[[a,b],c],d] + [[[b,c],d],a] + [[[c,d],a],b] + [[[d,a],b],c],
  \]
and Filippov's $h$-identity \cite{F0} for the 7-dimensional simple non-Lie Malcev algebra,
  \begin{align*}
  &
  [[[[a,b],c],d],e]
  + [[[[a,b],c],e],d]
  - [[[[a,b],d],c],e]
  - [[[[a,b],e],c],d]
  \\
  &
  + [[[[a,d],b],e],c]
  - [[[[a,d],e],b],c]
  + [[[[a,e],b],d],c]
  - [[[[a,e],d],b],c]
  \\
  &
  +2 [[[a,b],[c,d]],e]
  +2 [[[a,b],[c,e]],d]
  +2 [[[a,d],[b,e]],c]
  +2 [[[a,e],[b,d]],c]
  \equiv 0.
  \end{align*}
We introduce three independent identities in degree 5 relating the bilinear and trilinear products:
  \begin{equation} \label{3new5}
  \left\{
  \begin{array}{l}
   2  [ [ \{ a, c, d \}, b ], e ]
   +  [ [ \{ a, c, d \}, e ], b ]
  -2  [ [ \{ a, e, d \}, b ], c ]
   -  [ [ \{ a, e, d \}, c ], b ]
   \\[3pt]
   +  [ \{ a, c, d \}, [ b, e ] ]
   -  [ \{ a, e, d \}, [ b, c ] ]
   +  \{ a, [ [ b, c ], e ], d \}
   -  \{ a, [ [ b, e ], c ], d \}
   \\[3pt]
  -2  \{ a, [ [ c, e ], b ], d \}
  \equiv 0,
  \\[6pt]
   2  \{ [ [ b, c ], e ], a, d \}
  -2  \{ [ [ b, d ], e ], a, c \}
   +  \{ [ [ b, e ], c ], a, d \}
   -  \{ [ [ b, e ], d ], a, c \}
   \\[3pt]
  +2  \{ [ [ c, d ], e ], a, b \}
   -  \{ [ [ c, e ], b ], a, d \}
   +  \{ [ [ c, e ], d ], a, b \}
   +  \{ [ [ d, e ], b ], a, c \}
   \\[3pt]
   -  \{ [ [ d, e ], c ], a, b \}
  \equiv 0,
  \\[6pt]
      \{ [ [ a, d ], e ], c, b \}
   -  \{ [ [ a, e ], d ], c, b \}
   -  \{ [ [ c, d ], e ], a, b \}
   +  \{ [ [ c, e ], d ], a, b \}
   \\[3pt]
  -2  \{ [ [ d, e ], a ], c, b \}
  +2  \{ [ [ d, e ], c ], a, b \}
  +2  \{ b, [ [ a, c ], d ], e \}
  -2  \{ b, [ [ a, c ], e ], d \}
   \\[3pt]
   +  \{ b, [ [ a, d ], c ], e \}
   -  \{ b, [ [ a, e ], c ], d \}
   -  \{ b, [ [ c, d ], a ], e \}
   +  \{ b, [ [ c, e ], a ], d \}
  \equiv 0.
  \end{array}
  \right.
  \end{equation}
Recall that the set $\mathsf{Sh}_{n_1 \cdots n_k}$ of $(n_1,\dots,n_k)$-shuffles where $n_1+\cdots+n_k=n$ consists of all permutations $\sigma$
of the ordered set $X = \{x_1,\dots,x_n\}$ (where $x_i \prec x_j \!\iff\! i < j$) satisfying
$x_{n_1+\cdots+n_i+1}^\sigma \prec \cdots \prec x_{n_1+\cdots+n_{i+1}}^\sigma$
for all $i = 1, \dots, k{-}1$.
We introduce three new independent identities in degree 6 relating the bilinear and trilinear products;
in each case the set $X$ of permuted variables is indicated by the superscript $\sigma$:
  \begin{equation} \label{2new6}
  \left\{
  \begin{array}{l}
  \displaystyle{\sum_{\sigma \in \mathsf{Sh}_{212}}}
  \{ \, [ \, [ b^\sigma, c^\sigma ], d^\sigma ], a, [ e^\sigma, f^\sigma ] \}
  \equiv 0,
  \\[6pt]
  \displaystyle{\sum_{\sigma \in \mathsf{Sh}_{21}}}
  \Big(
   3  [ \{ a, [ c^\sigma, d^\sigma ], e \}, [ b, f^\sigma ] ]
  +3  [ \{ a, [ [ c^\sigma, d^\sigma ], b ], e \}, f^\sigma ]
  \\[-6pt]
  \quad\quad\quad\quad
  -3  [ [ c^\sigma, d^\sigma ], \{ a, [ b, f^\sigma ], e \} ]
  -3  [ [ [ c^\sigma, d^\sigma ], b ], \{ a, f^\sigma, e \} ]
  \\
  \quad\quad\quad\quad\quad
  -   [ \{ a, [ [ c^\sigma, d^\sigma ], f^\sigma ], e \}, b ]
  -   [ \{ a, b, e \}, [ [ c^\sigma, d^\sigma ], f^\sigma ] ]
  \Big)
  \equiv 0.
  \end{array}
  \right.
  \end{equation}
For the third new identity in degree 6, see Figure \ref{newdegree6}.

Altogether we found five new independent identities in degree 6 using modular arithmetic ($p = 103$ and $\sqrt{2} \equiv 38$).
We then used the Maple function \texttt{iratrecon} to determine the simplest rational numbers corresponding to each modular coefficient.
This worked for the first four identities, but failed for the fifth.
We therefore ran the Maple worksheet again with a much larger prime ($p = 100049$ and $\sqrt{2} \equiv 10948$);
this time the rational reconstruction was successful for all five identities.
Table \ref{coefftable} gives the number of terms in each identity, the integer coefficients, the squared Euclidean length of the coefficient vector,
and the dimension of the $S_6$-module generated by the lifted identities and the new identities up to the current identity.
For each identity, we have multiplied the rational coefficients by the least common multiple of their denominators, and divided the resulting integers
by their greatest common divisor.
The new identities are sorted by increasing length.
We used the integer coefficients to check the new identities using arithmetic in characteristic 0 with $\sqrt{2}$:
for each identity, we performed ten trials by generating six random elements of $LJY_3$ with two-digit decimal coefficients,
substituting these elements into the identity, verifying that the result evaluated to 0 after simplification.

  \begin{figure}
  \begin{alignat*}{4}
  {} -9 & \{ e{,} [ [ [ b{,} d ]{,} a ]{,} c ]{,} f \} &
  {}  -9 & [ [ \{ c{,} b{,} d \}{,} e ]{,} [ a{,} f ] ] &
  {}  -8 & \{ c{,} [ [ a{,} f ]{,} [ b{,} e ] ]{,} d \} &
  {}  -8 & [ [ \{ e{,} d{,} f \}{,} b ]{,} [ a{,} c ] ] \\
  {}  -6 & [ \{ c{,} b{,} d \}{,} [ [ e{,} f ]{,} a ] ] &
  {}  -6 & [ [ [ \{ e{,} d{,} f \}{,} c ]{,} b ]{,} a ] &
  {}  -4 & \{ c{,} [ [ [ a{,} f ]{,} b ]{,} e ]{,} d \} &
  {} -4 & [ \{ e{,} d{,} f \}{,} [ [ b{,} c ]{,} a ] ] \\
  {} -4 & [ \{ e{,} [ [ b{,} d ]{,} c ]{,} f \}{,} a ] &
  {} -4 & [ [ \{ e{,} d{,} f \}{,} [ a{,} c ] ]{,} b ] &
  {} -3 & \{ e{,} [ [ [ b{,} d ]{,} c ]{,} a ]{,} f \} &
  {} -3 & \{ e{,} [ [ [ a{,} d ]{,} b ]{,} c ]{,} f \} \\
  {} -3 & [ \{ c{,} b{,} d \}{,} [ [ a{,} e ]{,} f ] ] &
  {} -3 & [ [ \{ c{,} e{,} d \}{,} f ]{,} [ a{,} b ] ] &
  {} -3 & [ \{ [ [ d{,} f ]{,} c ]{,} b{,} e \}{,} a ] &
  {} -3 & [ \{ [ [ c{,} f ]{,} d ]{,} b{,} e \}{,} a ] \\
  {} -2 & \{ c{,} [ [ [ e{,} f ]{,} a ]{,} b ]{,} d \} &
  {} -2 & \{ c{,} [ [ [ b{,} f ]{,} a ]{,} e ]{,} d \} &
  {} -2 & [ \{ e{,} d{,} f \}{,} [ [ a{,} b ]{,} c ] ] &
  {} -2 & [ \{ c{,} f{,} d \}{,} [ [ b{,} e ]{,} a ] ] \\
  {} -2 & [ [ \{ c{,} e{,} d \}{,} a ]{,} [ b{,} f ] ] &
  {} -2 & [ \{ e{,} [ [ c{,} d ]{,} b ]{,} f \}{,} a ] &
  {} -2 & [ \{ e{,} [ [ b{,} c ]{,} d ]{,} f \}{,} a ] &
  {} -2 & [ \{ [ [ d{,} e ]{,} f ]{,} b{,} c \}{,} a ] \\
  {} -2 & [ \{ [ [ c{,} e ]{,} f ]{,} b{,} d \}{,} a ] &
  {} -2 & [ [ \{ c{,} f{,} d \}{,} [ b{,} e ] ]{,} a ] &
  {} -  & [ \{ c{,} f{,} d \}{,} [ [ a{,} b ]{,} e ] ] &
  {} -  & [ \{ c{,} e{,} d \}{,} [ [ a{,} f ]{,} b ] ] \\
  {} -  & [ [ \{ c{,} f{,} d \}{,} b ]{,} [ a{,} e ] ] &
  {} -  & [ \{ [ [ d{,} f ]{,} e ]{,} b{,} c \}{,} a ] &
  {} -  & [ \{ [ [ c{,} f ]{,} e ]{,} b{,} d \}{,} a ] &
  {} +  & [ \{ c{,} e{,} d \}{,} [ [ a{,} b ]{,} f ] ] \\
  {} +  & [ \{ [ [ e{,} f ]{,} d ]{,} b{,} c \}{,} a ] &
  {} +  & [ \{ [ [ e{,} f ]{,} c ]{,} b{,} d \}{,} a ] &
  {} +2 & \{ c{,} [ [ [ e{,} f ]{,} b ]{,} a ]{,} d \} &
  {} +2 & \{ c{,} [ [ [ b{,} f ]{,} e ]{,} a ]{,} d \} \\
  {} +2 & [ \{ c{,} e{,} d \}{,} [ [ b{,} f ]{,} a ] ] &
  {} +2 & [ [ \{ c{,} f{,} d \}{,} a ]{,} [ b{,} e ] ] &
  {} +2 & [ [ \{ e{,} d{,} f \}{,} [ b{,} c ] ]{,} a ] &
  {} +2 & [ [ \{ c{,} e{,} d \}{,} [ b{,} f ] ]{,} a ] \\
  {} +3 & \{ e{,} [ [ [ c{,} d ]{,} b ]{,} a ]{,} f \} &
  {} +3 & \{ e{,} [ [ [ a{,} d ]{,} c ]{,} b ]{,} f \} &
  {} +3 & [ \{ c{,} f{,} d \}{,} [ [ a{,} e ]{,} b ] ] &
  {} +3 & [ \{ c{,} b{,} d \}{,} [ [ a{,} f ]{,} e ] ] \\
  {} +3 & [ [ \{ c{,} f{,} d \}{,} e ]{,} [ a{,} b ] ] &
  {} +3 & [ [ \{ c{,} e{,} d \}{,} b ]{,} [ a{,} f ] ] &
  {} +4 & \{ c{,} [ [ a{,} e ]{,} [ b{,} f ] ]{,} d \} &
  {} +4 & \{ c{,} [ [ a{,} b ]{,} [ e{,} f ] ]{,} d \} \\
  {} +4 & \{ c{,} [ [ [ a{,} f ]{,} e ]{,} b ]{,} d \} &
  {} +4 & [ [ \{ e{,} d{,} f \}{,} a ]{,} [ b{,} c ] ] &
  {} +4 & [ [ \{ c{,} f{,} d \}{,} [ a{,} e ] ]{,} b ] &
  {} +6 & [ [ \{ e{,} d{,} f \}{,} c ]{,} [ a{,} b ] ] \\
  {} +6 & [ [ \{ c{,} b{,} d \}{,} a ]{,} [ e{,} f ] ] &
  {} +6 & [ [ [ \{ c{,} b{,} d \}{,} e ]{,} f ]{,} a ] &
  {} +9 & \{ e{,} [ [ [ c{,} d ]{,} a ]{,} b ]{,} f \} &
  {} +9 & [ [ \{ c{,} b{,} d \}{,} f ]{,} [ a{,} e ] ] \\
  {}+12 & \{ e{,} [ [ a{,} d ]{,} [ b{,} c ] ]{,} f \} &
  {}+12 & \{ e{,} [ [ [ b{,} c ]{,} d ]{,} a ]{,} f \}
  \\[-24pt]
  \end{alignat*}
  \caption{Third new multilinear identity in degree 6 for $LJY_3$}
  \label{newdegree6}
  \end{figure}

\begin{table}[h]
\begin{tabular}{rcrr}
terms & coefficients & length$^2$ & dimension \\ \midrule
 30 &\quad $\pm 1$ &\quad 30 &\quad 2632 \\
 18 &\quad $\pm 1, \pm 3$ &\quad 114 &\quad 2647 \\
 58 &\quad $\pm 1, \pm 2, \pm 3, \pm 4, \pm 6, -8, \pm 9, 12$ &\quad 1244 &\quad 2701 \\
333 &\quad $\pm 1, \pm 2, \pm 3, \pm 4, \pm 5, \pm 6, \pm 8, \pm 9, \pm 12$ &\quad 7468 &\quad 2732 \\
635 &\quad $\pm 3, \pm 9, \pm 18, \pm 27, \pm 36, \pm 45, \pm 104$ &\quad 1172619 &\quad 2733
\end{tabular}
\bigskip
\caption{Five new identities in degree 6 for $LJY_3$}
\label{coefftable}
\end{table}

\begin{theorem}
Every multilinear identity of degree $\le 6$ satisfied by $LJY_3$ is a consequence of the skew-symmetry of $[-,-]$, the symmetry of $\{-,-,-\}$,
the multilinear form of the Malcev identity, the Filippov $h$-identity, the three identities \eqref{3new5}, the two identities \eqref{2new6},
the identity in Figure \ref{newdegree6}, and two more multilinear identities in degree 6 with 333 and 635 terms.
\end{theorem}

We note that $LJY_3$ does not satisfy the identity $[[x,y],z] \equiv \{x,y,z\} - \{y,x,z\}$ from the definition of Lie-Jordan algebras.
On the other hand, the Malcev identity is a consequence of the Lie-Jordan  identities of degree $\le 4$;
in fact, it is a consequence of the liftings of $[[x,y],z] \equiv \{x,y,z\} - \{y,x,z\}$.

\subsection{The LJY algebra $LJY_4$}

Similar computations to those described in Section \ref{sectionresults} give the following result.

\begin{theorem}
Every multilinear polynomial identity of degree $\le 6$ satisfied by the LJY algebra $LJY_4$ is a consequence of
the skew-symmetry of the bilinear operation and the symmetry of the trilinear operation.
\end{theorem}

%%%%%%%%%%%%%%%%%%%%%%%%%%%%%%%%%%%%%%%%%%%%%%%%%%%%%%%%%%%%%%%%%%%%%%%%%%%%%%%%%%%%%%%%%%%%%%%%%%%%%%%%%%%%
%%%%%%%%%%%%%%%%%%%%%%%%%%%%%%%%%%%%%%%%%%%%%%%%%%%%%%%%%%%%%%%%%%%%%%%%%%%%%%%%%%%%%%%%%%%%%%%%%%%%%%%%%%%%
%%%%%%%%%%%%%%%%%%%%%%%%%%%%%%%%%%%%%%%%%%%%%%%%%%%%%%%%%%%%%%%%%%%%%%%%%%%%%%%%%%%%%%%%%%%%%%%%%%%%%%%%%%%%

\section{Conclusion} \label{sectionconclusion}

The class of Lie-Yamaguti algebras has not been broadly studied, probably because of its complex structure theory.
It would be interesting to work out some more explicit examples to help clarify the situation.
Here are some ideas to consider.

We could extend the computations in this paper to higher degrees,
study other families of reductive pairs using the methods of this paper,
including those producing $k$-irreducible Lie-Yamaguti algebras for $k \ge 2$.

In the spirit of Section~\ref{sectionjordan}, we can replace the LY products by other operations to see
which other structures can be defined on the underlying vector spaces of the representations of the Lie algebras.
As mentioned in \S~\ref{subsection:LY-representation},
if $\mathfrak{h}$ is a simple Lie algebra and $\mathfrak{m}$ is an irreducible $\mathfrak{h}$-module, and
the exterior square $\Lambda^2 \mathfrak{m}$ contains both $\mathfrak{m}$ and the adjoint $\mathfrak{h}$-module,
then there are $\mathfrak{h}$-module morphisms
$ \beta\colon \Lambda^2 \mathfrak{m} \rightarrow \mathfrak{m} $ and $\tau\colon \Lambda^2 \mathfrak{m} \rightarrow \mathfrak{h}$,
which give bilinear and trilinear products on $\mathfrak{m}$ defined by
$[x,y] = \beta(x,y)$ and $ \langle x,y,z \rangle = \tau(x,y) \cdot z$;
then a Lie algebra structure can be defined on $\mathfrak{h} \oplus \mathfrak{m}$ if and only if
identities (LY3) and (LY4) are satisfied.
Thus we need to find the polynomial identities of low degrees satisfied by these operations
to see if we get LY-structures or different classes of algebras.

We used the online LiE software \cite{LiE}
to study the Lie algebra of type $A_2 = \mathfrak{sl}_3(\mathbb{C})$.
The adjoint representation has weight $\lambda_1 + \lambda_2$, denoted $[1,1]$ in LiE.
All the examples we found of modules whose exterior squares contain both the original module and the adjoint module satisfy $a = b$
(and hence are self-dual), which suggests the conjecture that $\Lambda^2 [a,b]$
contains $[a,b]$ and $[1,1]$ if and only if $a = b$.

\smallskip

\noindent $\bullet$
For $[1,1]$, the projection onto $[1,1]$ gives the Lie bracket on the adjoint module, so this is a Lie-Yamaguti algebra
of adjoint type:
  \begin{equation*}
  \label{[1,1]}
  \Lambda^2 [1,1] = [3,0] \oplus [0,3] \oplus \fbox{$[1,1]$}
  \end{equation*}

\smallskip

\noindent $\bullet$
For $[2,2]$, all multiplicities are 1 in the decomposition:
  \begin{equation*}
  \label{[2,2]}
  \begin{array}{l}
  \Lambda^2 [2,2]
  =
  [5,2] \oplus  [2,5] \oplus  [3,3] \oplus  [4,1] \oplus  [1,4] \oplus
  \fbox{$[2,2]$} \oplus  [3,0] \oplus  [0,3] \oplus  \fbox{$[1,1]$}
  \end{array}
  \end{equation*}
This produces an LY algebra that appears in references \cite{BEM09} and \cite{BEM11}.

\smallskip

\noindent $\bullet$
For $[3,3]$, the adjoint module $[1,1]$ occurs with multiplicity 1, but $[3,3]$ occurs with multiplicity 2, so there
will be a one-parameter family of different bilinear products to consider:
  \[
  \begin{array}{l}
  \Lambda^2 [3,3]
  =
  [7,4] \oplus  [4,7] \oplus  [5,5] \oplus  [9,0] \oplus  [6,3] \oplus  [3,6] \oplus [0,9] \oplus  [7,1] \oplus  [4,4] \oplus  [1,7]
  \\
  \qquad\quad
  {}
  \oplus 2[5,2] \oplus 2[2,5] \oplus \fbox{$2[3,3]$} \oplus  [4,1] \oplus  [1,4] \oplus  [2,2] \oplus  [3,0] \oplus  [0,3] \oplus \fbox{$[1,1]$}
  \end{array}
  \]

\smallskip

\noindent $\bullet$
For $[4,4]$ the multiplicities are the same as in the previous case:
  \[
  \begin{array}{l}
  \Lambda^2 [4,4]
  =
  [ 9, 6] \oplus  [ 6, 9] \oplus  [ 7, 7] \oplus  [11, 2] \oplus  [ 8, 5] \oplus  [ 5, 8] \oplus [ 2,11] \oplus  [ 9, 3] \oplus  [ 6, 6]
  \\
  \qquad\quad
  {}
  \oplus  [ 3, 9] \oplus  [10, 1] \oplus 2[ 7, 4] \oplus 2[ 4, 7] \oplus  [ 1,10] \oplus  [ 8, 2] \oplus 2[ 5, 5] \oplus  [ 2, 8] \oplus  [ 9, 0]
  \\
  \qquad\quad
  {}
  \oplus 2[ 6, 3] \oplus 2[ 3, 6] \oplus  [ 0, 9] \oplus  [ 7, 1] \oplus \fbox{$2[ 4, 4]$} \oplus  [ 1, 7] \oplus 2[ 5, 2] \oplus 2[ 2, 5]
  \\
  \qquad\quad
  {}
  \oplus 2[ 3, 3] \oplus  [ 4, 1] \oplus  [ 1, 4] \oplus  [ 2, 2] \oplus {} [ 3, 0] \oplus  [ 0, 3] \oplus  \fbox{$[ 1, 1]$}
  \end{array}
  \]

%%%%%%%%%%%%%%%%%%%%%%%%%%%%%%%%%%%%%%%%%%%%%%%%%%%%%%%%%%%%%%%%%%%%%%%%%%%%%%%%%%%%%%%%%%%%%%%%%%%%%%%%%%%%
%%%%%%%%%%%%%%%%%%%%%%%%%%%%%%%%%%%%%%%%%%%%%%%%%%%%%%%%%%%%%%%%%%%%%%%%%%%%%%%%%%%%%%%%%%%%%%%%%%%%%%%%%%%%
%%%%%%%%%%%%%%%%%%%%%%%%%%%%%%%%%%%%%%%%%%%%%%%%%%%%%%%%%%%%%%%%%%%%%%%%%%%%%%%%%%%%%%%%%%%%%%%%%%%%%%%%%%%%

\end{document}